\documentclass[intlimits]{amsart}
\usepackage[utf8]{inputenc} 
\usepackage{epigraph}
\usepackage[all]{xy}
\usepackage{amssymb}
\usepackage{mathtools}
\usepackage{hyperref}

\newcommand{\lcm}{\text{lcm}}

\newcommand{\rad}{{\rm rad}}

\newcommand{\diag}{{\rm diag}}

\newcommand{\rk}{{\rm rk}}

\newcommand{\Q}{{\bf Q}}

\newcommand{\0}{\boldsymbol{0}}
\newcommand{\Z}{\mathbb{Z}}
\newcommand{\C}{\mathbb{C}}

\newcommand{\N}{\mathbb{N}}
\newcommand{\Hn}{\mathfrak{H}_n}

\newcommand{\vtn}{\vartheta^{(n)}}
\newcommand{\gen}{\operatorname*{gen}}
\newcommand{\A}{\mathbb{A}}
\newcommand{\R}{\mathbb{R}}

\newcommand{\tr}{{\rm tr}}

\newcommand{\mcB}{{\mathcal B}}

\newtheorem{theorem}{Theorem}[section]
\newtheorem{lemma}[theorem]{Lemma}
\newtheorem{definition}[theorem]{Definition}
\newtheorem{remark}[theorem]{Remark}

\newtheorem{corollary}[theorem]{Corollary}

\newtheorem{proposition}[theorem]{Proposition}
\numberwithin{equation}{section}

\begin{document}
`
\setcounter{MaxMatrixCols}{20}

\title
{Paramodular groups and theta series}  
\author[S. Böcherer, R. Schulze-Pillot]{Siegfried Böcherer, Rainer Schulze-Pillot} 

\maketitle

\begin{abstract}
  For a paramodular group of any degree and square free level we study
  the Hecke algebra and the boundary components. We define paramodular
  theta series and show that for square free level and large enough
  weight they generate the space of cusp forms (basis problem), using
  the doubling or pullback of Eisenstein series method. For this we
  give a new geometric proof of Garrett's double coset decomposition
  which works in our more general situation.
\end{abstract}
\section{Introduction}\label{sec1}
Paramodular groups are rational discrete subgroups of the real symplectic group
which have been considered in \cite{scc,scs, christian} and treated by
Siegel under the name ``Stufengruppe der 
Stufe $T$'' in \cite{siegel_stufe}. From an adelic point of view a maximal paramodular group is a
group of type $Sp_n(\Q)\cap \prod_{p \text{ prime}}K_p$, where
$K_p\subseteq Sp_n(\Q_p)$ is a maximal compact  subgroup which is not
special for all primes $p$  dividing
the determinant of the defining matrix $T \in M_n^{{\rm sym}}(\Z)\cap
GL_n(\Q)$ and special for all other primes $p$. The maximal
paramodular groups are also called paramodular groups of square free
level since they can also be characterized by the condition that
$NT^{-1}$ is integral for some square free integer $N$; notice that
the concept of level used here is quite different from the one for
congruence subgroups of the integral symplectic group.

The theory of Siegel modular forms for paramodular groups has recently gained  increasing interest, in particular
in the case of degree $2$. Reasons for this are the investigations of
Gritsenko and Hulek of moduli spaces of abelian surfaces with
$(1,t)$-polarization, see e. g. \cite{gr-hu_imrn99}, the theory of
newforms for the group $GSp(4)$  (written as $GSp_2$ in our
terminology) of Roberts and Schmidt
\cite{roberts_schmidt}, and the modularity conjecture of Brumer and
Kramer \cite{brum-kra} together with the numerical evidence provided
by Poor and Yuen \cite{poor_yuen_paramodular} and the partial proof in
\cite{berger_klosin}, see also
\cite{fretwell} for the connection to a conjecture of Harder about
congruences. 

In contrast to the theory of Siegel modular forms for
the modular group $Sp_n(\Z)$ and its congruence subgroups of type
$\Gamma_0^{(n)}(N)$, however, very little is known for paramodular
groups of higher degree $n>2$. In particular,
the methods of Poor and Yuen for constructing
concrete examples of paramodular forms in degree $2$ are based on Gritsenko's lift
generalizing the classical lift of Maa\ss\  from Jacobi forms to Siegel
modular forms of degree $2$ and can not be adapted to cases of higher
degree without first finding a substitute for Gritsenko's lift in
those cases. Theta series of positive definite quadratic forms as used for
the modular group and its congruence subgroups don't have a good
transformation behaviour under paramodular groups.  In some cases they could be used by
Poor and Yuen with the help of tracing constructions, but a definition
of theta series that is adapted to the paramodular situation has been
missing so far. 
Moreover, the Hecke algebras of
paramodular groups, even if they are maximal, can not be treated using the general results for
Hecke algebras of special
maximal compact subgroups of $p$-adic reductive groups
\cite{satake}\cite[Sec. 3.5]{cartier}
since some of
their $p$-adic components are not special. This is probably the reason
why they have so far been studied only for  degree $2$ by
Gallenkämper and Krieg in
\cite{gall_krieg}.

The goal of this article is to study the Hecke algebras of maximal paramodular
groups in spite of these obstacles and to show that in the theory of paramodular
forms of any degree suitably defined theta series
of positive definite quadratic forms can play a similar role as the
usual theta series do for modular forms  for the full modular group
$Sp_n(\Z)$ or its congruence subgroups of type
$\Gamma_0^{(n)}(N)$. In
particular we give a definition of paramodular theta series and study
in the case of square free level the so called 
basis problem, i.e., the question 
whether these theta series generate the full space of paramodular forms of the
type in question and can hence be used for the explicit construction
of a basis of this space.  As an important tool for this we  give explicit
representatives of the basic double cosets in Hecke
algebras of maximal paramodular groups of arbitrary degree, generalizing the result of
\cite{gall_krieg}. An interesting feature of this Hecke algebra is
that, although it is not commutative, it contains a commutative
subalgebra which is very similar to the subalgebra generated by the
$T(m)$ in the theory for  the  modular group and can be used in our
applications to play the same role as that algebra does for the
modular group.   Our methods rely very much on the maximality of the
groups and we do not attempt to treat more general levels,
corresponding to non maximal groups. 
In fact we have no idea whether our results or
analogues thereof are true or false in such
more general cases.

In Section \ref{sec2} we study lattices in a vector space with nondegenerate
alternating bilinear form and their totally isotropic primitive
submodules. We use the lattice approach to define, following \cite{scc}, para-symplectic
groups as isometry groups of lattices and obtain
the usual integral paramodular groups as matrix groups of these with respect to
suitable bases. Using the study of  orbits of primitive totally
isotropic submodules we determine and count the boundary components of
the Siegel upper half space under the action of a paramodular group of
square free level and arbitrary degree.

In Section \ref{sec3} we study Hecke algebras for paramodular groups of square
free level and arbitrary degree and more general spaces of double cosets with respect to
two possibly different maximal paramodular groups.  As in the case of the full
integral symplectic group and its congruence subgroups the double
coset space  factors by strong 
approximation into a restricted tensor product of local spaces. The
key idea for the investigation of these is to
view a local double coset as an orbit of a lattice $L$ in $\Q_p^{2m}$
under the action of the isometry group of a standard lattice
$\Lambda$. In the classical case of level $1$ these orbits are
characterized by the elementary divisors of $L$ with respect to
$\Lambda$, which leads to the usual diagonal representatives of the
double cosets. In our case one has to consider in addition the elementary
divisors in the dual $\Lambda^\#$ of $\Lambda$ with respect to the
alternating form. In this way we obtain explicit sets of
representatives of these double cosets, generalizing results of
Gallenkämper and Krieg \cite{gall_krieg} for degree $2$,  and find a commutative
subalgebra of the Hecke algebra that is important for our applications
in the sequel. An analogous approach should work for the study
of the Hecke algebra of a (not necessarily special) maximal compact
local orthogonal group, but we don't pursue this further here. It may
also be possible to use our methods for the determination of the
structure of the Hecke algebra.

Section \ref{sec4} is devoted to a generalization of Garrett's decomposition
from \cite{garrett} of
the group $Sp_{m+n}(\Z)$ into double cosets with respect to the Siegel
maximal parabolic subgroup on one side and a diagonally  embedded
product $Sp_m(\Z)\times Sp_n(\Z)$ on the other side.
This decomposition lies at the heart of the doubling method or method
of pullbacks of Eisenstein series, which is an often used tool for the
study of the analytic properties of $L$-functions of Siegel modular
forms and for the study of the basis problem.

To generalize it to the paramodular
situation we have to replace  $Sp_m(\Z)\times Sp_n(\Z)$ by paramodular groups
$\Gamma_1=Sp(\Lambda_1)\subseteq Sp_m(\Q),
\Gamma_2=Sp(\Lambda_2)\subseteq Sp_n(\Q)$ of square free 
levels attached to suitable lattices $\Lambda_1\subseteq \Q^{2m},
\Lambda_2\subseteq \Q^{2n}$.  We found it difficult to generalize Garrett's ``tedious bit of 
linear algebra'' approach of
finding representatives by explicit elementary matrix operations and
then relate these to Hecke double coset representatives.

Instead, we choose a geometric approach. We identify  the double cosets with orbits of maximal totally
isotropic submodules $X$  of $\Lambda:=\Lambda_1\oplus \Lambda_2$ under the
action of $Sp(\Lambda_1)\times Sp(\Lambda_2)$ and investigate the
projections $X_1,X_2$ to $\Lambda_1,\Lambda_2$ of such a module (which
are in general no longer totally isotropic). This idea originates in
\cite{ps-rallis} and has been used by Murase \cite{murase} and
by Müller \cite{mueller}, who, in the case of the integral symplectic
group,  constructs explicit bases of the $X_i$ in order to find
representatives of Garrett's double cosets. Modifying that approach we
find that a triple consisting of the $Sp(\Lambda_i)$-orbits of the radicals $\rad(X_i)$
together with the orbit of a natural isomorphism
$\phi:X_1/\rad(X_1)\to X_2/\rad(X_2)$ under the action of suitable
paramodular 
subgroups $\Gamma, \Gamma'$ of the $Sp(\Lambda_i)$ characterizes the
orbit of $X$.
The orbits of the primitive totally isotropic submodules $\rad(X_i)$
of $\Lambda_i$ have been classified in Section \ref{sec2}, and the orbit of the
natural isomorphism $\phi$ can be described by a Hecke double coset
with respect to paramodular groups of degree $r=\rk(X_i/\rad(X_i))/2$.

In this way we obtain a direct proof that representatives of Garrett's
double cosets can be expressed in terms of
the representatives of Hecke double cosets with respect to
paramodular groups 
$\Gamma,\Gamma'$ and then use our
computation of the latter from the previous section in order to return
to explicit matrix representatives. In the situation
of the full modular group this gives a more conceptual
simplification of Garrett's proof. 

In Section \ref{sec5} we define paramodular theta series associated to lattice chains and
show that this gives examples of paramodular forms. In Section \ref{sec6} we prove a
paramodular version of Siegel's theorem, expressing  a weighted linear
combination of these theta series as a paramodular Eisenstein series.

Section \ref{sec7} then puts together the results obtained so far and
generalizes the approach from \cite{boech} to the basis problem
for the modular group to the paramodular situation, again only for
square free levels. We treat here only
the case of paramodular cusp forms and show that all paramodular cusp
forms of square free level can be written as linear combinations of
paramodular theta series of the same level.
We notice that it is this
level aspect that causes most of the technical difficulties in our
approach; we don't see a way to circumvent them by using the adelic
doubling method instead of our classical pullback method.

We have not been able to treat the
non-cuspidal case, which presents
some   interesting challenges caused by the more complicated structure 
of (equivalence classes of)
cusps, see Section \ref{sec2}.
In particular, we get only one Siegel Eisenstein series, but
for more general boundary 
components we have to consider Klingen Eisenstein series
for lower rank paramodular groups with  several 
different polarization matrices. To deal with these seems to be more
difficult than one might think at first sight.

 \section{Paramodular groups and their modular forms}\label{sec2}
 \begin{definition}\label{para-symplectic}
Let $R$ be a principal ideal domain with field of fractions $F$ and $\Lambda\subseteq V$ 
  an $R$-lattice on the $2m$-dimensional $F$-vector space $V$  with
  nondegenerate alternating bilinear form $\langle, \rangle$, assume
  $\langle \Lambda,\Lambda\rangle\subseteq R$.

We denote the group of
isometries  of the alternating module $(\Lambda,\langle,\rangle)$ by $Sp(\Lambda)$ and call it its symplectic or
para-symplectic group.

We call an $R$-basis
$\mcB=\{e_i,f_i\mid 1 \le i \le m\}$ of $\Lambda$  a
para-symplectic basis if one has $\langle e_i,e_j\rangle=\langle
f_i,f_j\rangle=0$, $\langle e_i,f_j\rangle =d_i\delta_{ij}$ with
$d_i\ne 0$ for all $i, j$. It is called ordered or of elementary
divisor type if one has $\langle e_i,f_i\rangle \mid \langle
e_{i+1},f_{i+1} \rangle$ for all $i<m$.

We call the group of matrices of the elements of $Sp(\Lambda)$ with respect to
such a para-symplectic basis an integral paramodular group of matrix level
(Matrixstufe) $T={\rm diag}(d_1,\ldots,d_m)$ and denote this group by
$\hat{\Gamma}^{(m)}(T)$. The
subgroup of the symplectic matrix group $Sp_m(F)\subseteq GL_{2m}(F)$
consisting of the matrices of the elements of $Sp(\Lambda)$ with
respect to the standard symplectic basis $\{e_i,d_i^{-1}f_i\mid 1\le i
\le m\}$ of $V$ is called a symplectic paramodular group or a
paramodular subgroup of $Sp_m(F)$ of matrix level (Matrixstufe) $T$
and will be denoted by $\Gamma^{(m)}(T)$. If all the $d_i$ are square
free we say that the matrix level $T$ is square free.

We
  say that $\Lambda$ has 
 level $N\in R$
 if $N$ is the least common multiple of the $d_i$ above, or equivalently if 
 $N$ is the least common multiple of all $N'$ for which $\Lambda^\#\supseteq \Lambda
  \supseteq N'\Lambda^\#$ holds, where $\Lambda^\#$ is the dual lattice with
  respect to $\langle, \rangle$.

The determinant of $\Lambda$ is the square class $\det(A)(R^\times)^2$
of the determinant of the matrix $A=(a_{ij})=(\langle v_i,v_j\rangle)$
of the bilinear form $\langle, \rangle$ with respect to a basis
  $\{v_i\}$ of $\Lambda$.
\end{definition}

\begin{remark}
  \begin{enumerate}
\item If $F$ is a nonarchimedean local field with ring of integers $R$
  and prime element $\pi$, the groups
  $\Gamma^{(m)}(\diag(1,\ldots,1,\pi,\ldots,\pi))$ with $1\le r< m$
  entries $\pi$ are stabilizers of non-special vertices in the
  Bruhat-Tits building of $Sp(m,F)$ and represent the conjugacy
  classes of maximal compact subgroups of $Sp_m(F)$ which are not
  special, see e. g. \cite{garrettbook}. 
\item A para-symplectic basis of elementary divisor type as described in the definition always
  exists \cite{scc}. 
\item Siegel \cite{siegel_stufe} defined for any nonsingular  matrix $T\in
M_{2m}(\Z)$ the ``Modulgruppe der Stufe $T$'' to be the
set of all matrices $M\in M_{2m}(\Z)$ with
\begin{equation*}
  {}^tM
  \begin{pmatrix}
    0&T\\-T&0
  \end{pmatrix}
M=\begin{pmatrix}
    0&T\\-T&0
  \end{pmatrix}.
\end{equation*}
If $T$ is diagonal this is an integral paramodular
group as defined above.
If $T$ is symmetric it is the matrix group attached
to a group $Sp(\Lambda)$ as above with respect to a more
general basis of $\Lambda$, hence conjugate to a group for diagonal
$T$ in elementary divisor form.
For diagonal $T$ (assumed to be in elementary
divisor form) Kappler \cite{kapplerdiss}
called such a group ``Siegelsche Stufengruppe $\Gamma(m,T)$ der Stufe $T$'' and
studied generators of these groups. Later work on such groups and
their modular forms also used the symplectic realization and
introduced the now common word ``paramodular group''. We summarize
the relations between the various realizations used in the literature:
One has 
\begin{equation*}
\Gamma^{(m)}(T)=Sp_{m}(\Q)\cap \begin{pmatrix}1&0\\0&T\end{pmatrix}
M_{2m}(\Z) \begin{pmatrix}1&0\\0&T^{-1}\end{pmatrix}
\end{equation*}
and 
\begin{eqnarray*}
  \Gamma^{(m)}(T)&=&
 \begin{pmatrix}
    1_m&0_m\\0_m&T
  \end{pmatrix} \hat{\Gamma}^{(m)}(T)
  \begin{pmatrix}
    1_m&0_m\\0_m&T^{-1}
  \end{pmatrix}\\
&=&
    \begin{pmatrix}
     T^{-1}&0_m\\
     0_m&T 
    \end{pmatrix}
 \begin{pmatrix}
    T&0_m\\0_m&1_m
  \end{pmatrix} \hat{\Gamma}^{(m)}(T)
  \begin{pmatrix}
    T^{-1}&0_m\\0_m&1_m
  \end{pmatrix}
                     \begin{pmatrix}
                       T&0_m\\
                       0_m&T^{-1}
                     \end{pmatrix}.
\end{eqnarray*}

\end{enumerate}
\end{remark}

Mainly to fix notation, we recall some definitions concerning Siegel modular 
forms: Let ${\mathfrak H}_n$ be Siegel's upper half space of degree $n$
consisting of complex symmetric matrices of size $n$ with positive definite
imaginary part. The real symplectic group $Sp_n({\mathbb R})$ 
acts on ${\mathfrak H}_n$ by
$$(M,Z)\longmapsto MZ:=M\cdot Z:= (AZ+B)(CZ+D)^{-1}$$  
with $M=\left(\begin{array}{cc} A&B\\
C & D\end{array}\right)$; this group then also acts on functions 
$f:{\mathfrak H}_n\longrightarrow {\mathbb C}$ by
$$(M,f) \longmapsto (f\mid_kM)(Z):= \det(CZ+D)^{-k}f(M\cdot Z)$$

\begin{definition}\label{modularformdef}
  Let $T=\diag(d_1,\ldots,d_n)\in M_n(\Z)$ be a diagonal matrix.
  A modular form for $\Gamma^{(n)}(T)$ of weight $k$ %
  is for $n \geq 2$ a holomorphic function
$f: \: {\mathfrak H}_n \longrightarrow \C$ satisfying
 \begin{equation*}
   (f|_k\gamma)(Z) := f(\gamma Z)\det(CZ+D)^{-k} = %
   f(Z)
 \end{equation*}
for all $\gamma = \begin{pmatrix} A & B\\ C&D 
\end{pmatrix} \in \Gamma^{(n)}(T)$,
$Z \in {\mathfrak H}_n$. It is a cusp form if for
all $g \in Sp_n(\Q)$  the Fourier coefficients
of $(f\vert_kg)$ are zero
at all degenerate matrices.
\end{definition}
Note that this definition makes sense not only for paramodular groups but
more generally for any group $\Gamma$ commensurable with $Sp(n,{\mathbb Z})$.
For general properties of such modular forms we refer to the standard literature
\cite{andrianov, freitagbook, Klibuch}, in particular, all the results
on modular forms for congruence subgroups from \cite[chapt.2]{andrianov} 
are also valid for
paramodular forms.  

\begin{remark}\label{remark_para}
  \begin{enumerate}
    \item If $T=\diag(d_1,\ldots,d_n)$ is in elementary divisor form,
      one can, as is usual, without loss of generality assume $d_1=1$.
\item The action of $\hat{\gamma} \in
\hat{\Gamma}^{(n)}(T)$ on the upper  half space ${\mathfrak H}_n$
described in \cite{siegel_stufe} is the same as the usual action of 
\begin{equation*}
\begin{pmatrix}
    T&0_n\\0_n&1_n
  \end{pmatrix} \hat{\gamma}
  \begin{pmatrix}
    T^{-1}&0_n\\0_n&1_n
  \end{pmatrix}.
\end{equation*}
\item It is sometimes useful to consider for a number $N\in \N$ prime to $\det(T)$ slightly more general the
  groups
$
\hat{\Gamma}_0^{(n)}(T,N)$ for the set of $\hat{\gamma} \in
\hat{\Gamma}^{(n)}(T)$ for which the lower left $n\times n$-block has
entries divisible by $N$ and
\begin{equation*}
  \Gamma_0^{(n)}(T,N):=
  \begin{pmatrix}
    1_n&0_n\\0_n&T
  \end{pmatrix} \hat{\Gamma}_0^{(n)}(T,N)
  \begin{pmatrix}
    1_n&0_n\\0_n&T^{-1}
  \end{pmatrix}.
\end{equation*}
Equivalently we have 
\begin{equation*}
\Gamma_0^{(n)}(T,N)=Sp_{n}(\Q)\cap \begin{pmatrix}1&0\\0&NT \end{pmatrix}
M_{2n}(\Z) \begin{pmatrix}1&0\\0&N^{-1}T^{-1}\end{pmatrix}
\end{equation*}

In particular, for $n=2, T=\begin{pmatrix}1&0\\0&t\end{pmatrix}$ we have

\begin{equation*}
\Gamma_{t,0}^{(2)}(N):= \Gamma_0^{(2)}(T,N)
\{\begin{pmatrix}*&t*&*&*\\ *&*&*&*/t
  \\N*&Nt*&*&*\\Nt*&Nt*&t*&*\end{pmatrix} \in Sp_2(\Q)\} ,
\end{equation*}
where the asterisks symbolize integral entries.

For $T = 1_n$ we have $\Gamma_0^{(n)}(T,N) = \Gamma_0^{(n)}(N)
\subseteq Sp_n(\Z)$.

\end{enumerate}
\end{remark}

In order to discuss the boundary components of the quotient of 
Siegel's space ${\mathfrak H}_n$ by a paramodular group we return to the
setting of  a module $\Lambda$  with alternating bilinear form over a principal
ideal domain $R$ with field of quotients $F$.
It is well known that any such module $\Lambda$ has an ordered  para-symplectic basis. We will
later need that for square free level such a basis can be adapted to any given primitive
totally isotropic submodule of $\Lambda$:

\begin{theorem}\label{isotropic_submodule_bases}
  Let $R$ be a principal ideal domain with field of fractions $F$
and $\Lambda\subseteq V$ 
  an $R$-lattice on the $2m$-dimensional $F$-vector space $V$  with
  nondegenerate alternating bilinear form $\langle, \rangle$, assume
  $\langle \Lambda,\Lambda\rangle\subseteq R$ and that the 
level of $\Lambda$ divides the square free element  $N\in R$.

Let $Z$ be a primitive
   submodule of rank $r\le m$ of $\Lambda$ (i.e., $Z$ is a direct
   summand of $\Lambda$, equivalently $FZ\cap \Lambda=Z$) which is
   totally isotropic with respect to $\langle,\rangle$.
   
Then there are a basis
  $(e_1,\ldots, e_r)$ of $Z$ and  vectors $f_1, \ldots,f_r
  \in \Lambda$  generating a totally isotropic submodule of $\Lambda$
  and satisfying $\langle e_i,f_j\rangle =d_i\delta_{ij}$ with $d_i\mid N$ such
  that 
\begin{equation*}
M:=M(Z):=\bigoplus_{i=1}^rRe_i\oplus
  \bigoplus_{i=1}^r Rf_i\end{equation*} can be split off
orthogonally in $\Lambda$, i.e., one has $\Lambda=M \perp
\Lambda'$ for some submodule $\Lambda'\subseteq \Lambda$.

Moreover, if $X\supseteq Z$ is a sublattice of $\Lambda$ with
$Z\subseteq \rad(X):=\{x\in X\mid \langle x,X\rangle=\{0\}\}$ one has
$X=Z\perp X'$ with 
$X':=X\cap \Lambda'$, and $X'$ is nondegenerate if and only if $Z=\rad(X)$ holds.

In particular, if $X=Z$ is maximal totally isotropic there exist a basis
  $(e_1,\ldots, e_m)$ of $X$ and  vectors $f_1, \ldots,f_m
  \in \Lambda$  generating a totally isotropic submodule of $\Lambda$
  and satisfying $\langle e_i,f_j\rangle =d_i\delta_{ij}$ with $d_i\mid N$ such
  that 
\begin{equation*}
\Lambda=\bigoplus_{i=1}^mRe_i\oplus
  \bigoplus_{i=1}^m Rf_i.\end{equation*} 
\end{theorem}
\begin{proof} This is a modification of \cite[Theorem 3.1]{rsp_hyperbolic}:
If $e\in Z$ is a primitive vector we have $\langle e,
\Lambda\rangle\supseteq \langle e, N\Lambda^\#\rangle =NR$, hence
$\langle e, \Lambda \rangle = dR$ for some $d \mid N$.

We have then $\langle e, d\Lambda^\#\cap \Lambda\rangle \subseteq dR$, and
if there were a prime $p\in R$ with $\langle e, d\Lambda^\#\cap
\Lambda \rangle \subseteq pdR$ we had $e \in
\Lambda \cap pd(d\Lambda^\#\cap\Lambda)^\#=p\Lambda+(pd\Lambda^\#\cap \Lambda)$. If
one had $p\nmid d$ this would imply $p\mid \langle e, \Lambda\rangle$,
which is a contradiction. On the other hand, $p\mid d$ implies
$d\Lambda^\# \cap \frac{1}{p}\Lambda\subseteq \Lambda$ since the level
of $\Lambda$ is square free, which contradicts the assumption that 
$e$ is primitive. 

We can  therefore find $f\in d\Lambda^\#\cap\Lambda$
with $\langle e,f \rangle =d$.
Since we have $\langle
Re+Rf, \Lambda \rangle\subseteq dR$ we can again split off the
$d$-modular hyperbolic plane $Re+Rf$ in $\Lambda$ with
orthogonal complement $\Lambda_1$.
In any case we consider  $Z':=Z\cap \Lambda_1$
with $\rk(Z')=\rk(Z)-1$, $Z'$ primitive in $\Lambda_1$,  and see that by induction we can obtain the vectors
$e_1,\ldots,e_r,f_1,\ldots,f_r$ as asserted.

Let now $x \in X$. We put
\begin{equation*}
  x':=x-\sum_{i=1}^r \frac{\langle x, f_i\rangle}{\langle e_i,f_i\rangle}e_i
\end{equation*}
and have $\langle x',f_i\rangle=0$ for $1\le i \le r$. Moreover, since
$\langle e_i,f_i\rangle =d\mid N$ if and only if $\langle \Lambda,
f_i\rangle =dR$ holds, we have $x' \in X\cap M^\perp$, with $x-x'
\in Z$, which gives $X=Z \perp (X\cap M^\perp)$ as asserted.

The rest of the assertion is obvious.
\end{proof}

Obviously, the $d_iR^\times$ for the elementary divisors $d_i$
associated to an ordered para-symplectic basis determine the isometry class of $(\Lambda,\langle,
\rangle)$ and one has $d_1\dots d_mR^\times=DR^\times$, where $D^2R^\times=(D(\Lambda))^2$
is the determinant 
of $\Lambda$. %
If the level of $\Lambda$ is square free the 
product $D=d_1\dots d_m$ determines already the isometry class:

\begin{lemma}\label{linear_algebra}
Let $R$ be a principal ideal domain with field of fractions $F$ and
$\Lambda$ an $R$-lattice of rank $2m$ on the $2m$-dimensional
$F$-vector space $V$, equipped with 
a nondegenerate alternating bilinear form $\langle, \rangle$.

Assume that $\Lambda$ with this form has square free level $N$,
let $(e_1,\ldots,e_m$, $f_1,\ldots, f_m)$ be a basis of $\Lambda$ for
which the alternating form has matrix $\bigl(
\begin{smallmatrix}
  0&A\\-{}^tA&0
\end{smallmatrix}\bigr)$.
Assume that the determinant of the matrix of $\langle, \rangle$ is
$D^2$ and that $\tilde{d}_1,\ldots, \tilde{d}_m$ are divisors of $N$ with $\tilde{d}_1\cdot
\dots \cdot \tilde{d}_m=D$. 

Then $M:=\oplus_iRe_i$ has a basis
$\{\tilde{e}_j=\sum_is_{ij}e_i\}$ and $M':=\oplus_iRf_i$ has a basis
$\{\tilde{f}_j=\sum_it_{ij}f_i\}$ with
$\langle \tilde{e}_i,\tilde{f}_j\rangle=\tilde{d}_i\delta_{ij}$. In
particular, the isometry class of $(\Lambda, \langle, \rangle)$ is
determined by $DR^\times=\det(A)R^\times$.

\end{lemma}
\begin{proof}
Use the elementary divisor theorem.
\end{proof}
\begin{corollary}\label{orbits_isotropic}
  With the notation as in  Theorem \ref{isotropic_submodule_bases}
  $d(Z):=d_1\dots d_r$ is uniquely determined up to units.
  
Two primitive totally isotropic submodules
$Z_1,Z_2$ of $\Lambda$ are in the same $Sp(\Lambda)$-orbit  if and only
if they have the same rank and
$d(Z_1)R^\times= d(Z_2)R^\times$ holds. 

If $D^2$ is the determinant of $\Lambda$ and $d\mid D$,
there exists an $Sp(\Lambda)$-orbit of primitive totally isotropic
submodules $Z=Z_d$ of $\Lambda$ of rank $u$ with $d=d(Z)$ if and only if one has $d
\mid N^u$ and $\frac{D}{d}\mid N^{m-u}$.

In particular, if $N\in R^\times$ holds, there is only one
$Sp(\Lambda)$-orbit of primitive totally isotropic submodules of rank
$u$ for each $u\le m=\rk(\Lambda)$.
\end{corollary}
\begin{proof}
  Obvious.
\end{proof}

\begin{corollary}\label{cusps_cosets_isotropic} Let $\Lambda$ be a
  lattice on $V$ of rank $2m$ and square free level $N$.
  \begin{enumerate}
  \item Let $U$ be a totally isotropic subspace of $V$ of dimension $u$, let
$P_U=P_u=\{g\in Sp(V) \mid gU=U\}$ be the maximal parabolic subgroup of
$Sp(V)$ fixing $U$. Let $\{g_1,\ldots,g_t\}$ denote a set of
representatives of the double cosets $Sp(\Lambda) g P_U$ with $g \in Sp(V)$.

Then the $Z_i=\Lambda \cap g_iU$ are a set of representatives of the $Sp(\Lambda)$-orbits of primitive  totally isotropic
submodules of rank $u$ of $\Lambda$.
\item Let $R=\Z$ and $F=\Q$ , let the determinant of $\Lambda$ be  $D^2=\prod_{p\mid
    N}p^{2\ell_p}$.

Then the number  of
boundary components of type ${\mathfrak H}_{m-u}$ in the Satake compactification  (or cusps in the sense of \cite{ibu_alternating})
of ${\mathfrak H}_m$
for $Sp(\Lambda)$ is equal to

  \begin{equation*}
    \prod_{p\mid N}(\min(u,m-u,\ell_p,m-\ell_p)+1).
  \end{equation*}
  In particular for $u=m$ there is only one zero dimensional cusp.
 \end{enumerate}
\end{corollary}
\begin{proof}
  By the previous corollary the totally isotropic subspaces of $V$ of
  fixed dimension are in a single $Sp(V)$-orbit, and assertion a)
  follows.
  For b), the number of boundary components of type ${\mathfrak H}_{m-u}$ of ${\mathfrak H}_m$
for $Sp(\Lambda)$  equals the number of
double cosets $Sp(\Lambda) g P_u$.

By the strong approximation theorem for $Sp(V)$ over $\Q$ it 
is enough to prove the assertion for the case that $N=p$ is a prime.
If (using the notation of Theorem \ref{isotropic_submodule_bases}) the determinant of
  the nondegenerate  module $M(U\cap \Lambda)$ is denoted by
  $p^{2s}$, the $Sp(\Lambda)$-orbit of $U$ is determined by $s$ by
  Corollary \ref{orbits_isotropic}, so we have to count the possible
  values of $s$. One must
  have $0\le s\le u$ and $s \le \ell_p$, and all  values of
  $s$ with $0\le s \le \min(u, \ell_p)$ occur if $\ell_p\le m-u$
  holds. In this case we have $\min(u,m-u,\ell_p,m-\ell_p)= \min(u,
  \ell_p)$, so the assertion is true in this case. If $\ell_p>m-u$ 
  holds, we must have $\ell_p-s \le m-u$, hence $s \ge \ell_p-(m-u)$
  and $\ell_p-(m-u)\le s \le \min(\ell_p,u)$, and all these value of
  $s$ occur, so that we obtain
  $1+m-\max(\ell_p,u)=1+\min(m-\ell_p,m-u)=1+\min(m-\ell_p,m-u,\ell_p,u)$
  values, which proves the assertion.
\end{proof}
\begin{remark}
  \begin{enumerate}
  \item   Alternatively we can use Proposition 3.6 of \cite{ibu_alternating}
  to determine the number of boundary components and
count the $\theta$- admissible  $\sigma \in S_r$
defined in \cite[Definition 3.3]{ibu_alternating} to obtain the
formula for the number of boundary components; notice, however, that
the condition (2) in that definition 
should also contain that one has $b_{t+1}(\sigma,\theta)=\{i \in X\mid
t+1 \le i \le t+1+\kappa'\}$  for some $\kappa'$
if $b_{t+1}(\sigma,\theta)\ne \emptyset$.
\item If we identify the elements of $Sp(V)$ with their matrices with
  respect to the symplectic basis $\{e_i,d_i^{-1}f_i\}$, the group
  $P_u$ becomes the parabolic subgroup of $Sp_m(\Q)$ with a block
  $0_{u,2m-u}$ in the lower left corner and we obtain representatives
  of the double cosets $\Gamma^{(m)}(\bigl(
    \begin{smallmatrix}
      d_1&&\\
      &\dots&\\
      &&d_m
    \end{smallmatrix}\bigr)) g P_u$.
    In particular, for $u=m-1$ our formula for the number of double 
    cosets coincides with the result of \cite{marschnerdiss}.
    \item By the corollary, as in the case of the full modular group
      the fact that there is only one zero dimensional cusp implies
      that a modular form for a paramodular group $\Gamma^{(n)}(T)$ of
      square free 
      matrix level $T$ is a cusp form if and only if its Fourier
      expansion has nonzero Fourier coefficients only at nondegenerate matrices.
  \end{enumerate}
\end{remark}
For later use we construct representatives of the orbits in the corollary. 
\begin{lemma}\label{matrixrepresentatives_cusps}
With notations as before let ${\mathcal B}=\{e_1,\ldots,e_m,f_1,\ldots f_m\}$ be a
para-symplectic basis of $\Lambda$ of elementary divisor type with $\langle
e_i,f_j\rangle={d_i}\delta_{ij}$ and
$d_i\mid d_{i+1}$, let $v_i=d_i^{-1}f_i$. For $0\le r<m$ let $U_0=U_0(r)=\sum_{i=r+1}^m\Q
e_i$ and $M(U_0\cap\Lambda)=U_0\cap\Lambda+\sum_{i=r+1}^m\Z f_i$, let
$d$ be such that there exists a totally 
isotropic subspace $U$of $V$ of dimension $u=m-r$ with $d(U\cap
\Lambda)=d$. Let $\tilde{d}_1,\ldots,\tilde{d}_m$ be such that $\tilde{d}_i
\mid \tilde{d}_{i+1}$ for $1\le i \le r-1$ and for $r+1\le i \le m-1$
with $d=\prod_{i=r+1}^m\tilde{d}_i$ and $\prod_{i=1}^m\tilde{d}_i=\prod_{i=1}^md_i$.  

Then there is $g =g({\mathcal B},d,u) \in Sp(V)$ with $d(gU_0\cap \Lambda)=d$ such that the matrix $\gamma$ of $g$ with respect
to the basis $\{e_1,\ldots,e_m,v_1,\ldots v_m\}$ of $V$ is $\gamma
=\gamma({\mathcal B},u,d)=\bigl(
\begin{smallmatrix}
  S&0\\0&{}^tS^{-1} 
\end{smallmatrix}\bigr)$, where $S=S({\mathcal B},u,d)\in SL_m(\Z)$ is congruent to a
permutation matrix $S_p=S_p({\mathcal B},u,d)$ modulo each prime $p\mid N$ (where we admit  possible
entries $-1$ in the permutation matrix). Moreover, $g$ can be chosen such that
the $\tilde{e}_i:=ge_i,\tilde{f}_i:=\tilde{d}_igv_i$ for $1\le i\le m$ form a basis of
$\Lambda$, with the last $u$ pairs $\tilde{e}_i,\tilde{f}_i$ forming a
basis of 
$M(gU_0\cap \Lambda)$.
The matrices $\gamma({\mathcal B},u,d)$ with $d\mid D, d\mid
N^{u},\frac{D}{d}\mid N^{m-u}$ form a set of representatives of the
double cosets $Sp(\Lambda) \tilde{\gamma}P_u$ and hence of the set of boundary
components of type ${\mathfrak H}_{m-u}$ of the Siegel upper half
space ${\mathfrak H}_m$.

Moreover, the totally isotropic submodules $g({\mathcal B},u,d)U_0(r)\cap
\Lambda$ of $\Lambda$ form a set of representatives of the
$Sp(\Lambda)$-orbits of primitive totally isotropic submodules of $\Lambda$.
\end{lemma}
\begin{proof}
For $p\mid N$ there is a permutation $\sigma_p \in S_m$ such that
$d((\sum_{i=r+1}^m\Q e_{\sigma_p(i)})\cap \Lambda)\in d\Z_p^\times$, and we
may choose the permutation so that for $1\le i \le r-1$ and
for $r+1\le i \le m-1$ one has $\langle
e_{\sigma_p(i)},f_{\sigma_p(i)}\rangle \Z_p \supseteq \langle
e_{\sigma_p(i+1)},f_{\sigma_p(i+1)}\rangle \Z_p$ .
We let
$S_p \in SL_m(\Z)$ be the associated matrix (with an entry $-1$ if
${\rm sgn}(\sigma_p)=-1$). If $N=p$ is a prime we set $S=S_p$, and
with $\gamma$ and $g$ as in the assertion we have
$gU_0=\sum_{i=r+1}^m\Q e_{\sigma_p(i)} $ and are done.
For composite $N$ we have different permutations for its prime
factors and have to combine them into a single $S$. Indeed, by the strong approximation theorem for
$SL_m$  we find $S\in SL_m(\Z)$ congruent to $S_p$ modulo $p$ for each
$p\mid N$. Then $g\in Sp(V)$ with matrix $\bigl(
\begin{smallmatrix}
  S&0\\0&{}^tS^{-1} 
\end{smallmatrix}\bigr)$ with respect
to the basis $\{e_1,\ldots,e_m,v_1,\ldots v_m\}$ of $V$ is as desired.
\end{proof}
\begin{remark} 
 For $r=0$ with $d=D$ we may set $S=1_m$, in the trivial case $r=m$, where
 $U_0=\{\bf 0\}$, we set $d=1$ and
 $S=1_m$.
\end{remark}
\section{Hecke algebras for paramodular groups}\label{sec3}
We want to study Hecke algebras associated to paramodular groups and
slightly more general spaces of double cosets $\Gamma g \Gamma'$ where
$\Gamma, \Gamma'$ are possibly different paramodular groups of the
same degree.

We notice first that by \cite[4.7]{borel_jacquet} for any group $\Gamma
\subseteq Sp_n(\Q)$ which can be obtained as $\Gamma= Sp_n(\Q)\cap
\prod_pK_p$, where the $K_p$ are compact open subgroups of
$Sp_n(\Q_p)$  the Hecke algebra ${\mathcal H}(Sp_n(\Q),\Gamma)$
factors into a restricted tensor product of the local Hecke algebras
${\mathcal H}_p(Sp_n(\Q_p),K_p)$. This is proved using the strong
approximation theorem for the group $Sp_n$, and the argument carries
over to the situation of spaces of double cosets $\Gamma g \Gamma'$,
where $\Gamma,\Gamma'$ are groups of this type. In particular, the
result holds for the case of a pair of paramodular groups.

We can hence restrict attention to the study of the local double
coset spaces. We therefore let now $R$ be a
complete discrete valuation ring with prime element $p$ and field of
fractions $F$ and $V$ a 
vector space of dimension $2n$ over $F$ equipped with a nondegenerate
alternating bilinear form $\langle, \rangle$. A lattice $\Lambda$ on
$V$ of level dividing $p$ (also called a $p$-elementary lattice in the
sequel)  has then for some uniquely determined integers $a,b$ with 
$a+b=n$ a para-symplectic basis of
elementary divisor type 
$\{e_1,\ldots,e_n,f_1,\ldots f_n\}$ with $\langle e_i,f_i\rangle =1$
for $1\le i \le a$, $\langle e_i,f_i\rangle =p$ for $a+1\le i \le  n$. We
say then that the canonical decomposition of $\Lambda$ has $a$ unimodular hyperbolic planes
$Re_i+Rf_i$ and $b$ $p$-modular hyperbolic planes.
\begin{proposition}\label{hecke_latticeversion}%
  Let $\Lambda$ be a lattice of level dividing $p$ on $V$
and $e_1,\ldots,e_n,f_1,\ldots,f_n$ a para-symplectic basis of $\Lambda$ with  $\langle e_i,f_j \rangle
 =\delta_{ij}$ for $1\le i \le a$, $\langle e_i,f_j\rangle
 =p\delta_{ij}$ for $a+1 \le i \le n=a+b$.
  For $0\le a'\le n$ and $b'=n-a'$ let   \begin{equation*}\begin{split}
    L(r_+&,r_-,\mu_1,\ldots,\mu_n):=\\
      & \Z_pp^{\mu_{a+1}}e_{a+1}+
    \Z_pp^{-\mu_{a+1}-1}f_{a+1}+\dots+\Z_pp^{\mu_{a+r_+}}e_{a+r_+}+
    \Z_pp^{-\mu_{a+r_+}-1}f_{a+r_+}\\&
  +\Z_pp^{\mu_{r_-+1}}e_{r_-+1}+\Z_pp^{-\mu_{r_-+1}}f_{r_-+1}+\dots+\Z_pp^{\mu_a}e_a+\Z_pp^{-\mu_a}f_a\\
& + \Z_pp^{\mu_1}e_1+
\Z_pp^{-\mu_1+1}f_1+\dots+\Z_pp^{\mu_{r_-}}e_{r_-}+\Z_pp^{-\mu_{r_-}+1}f_{r_-}\\
&+\Z_pp^{\mu_{a+r_++1}}e_{a+r_++1}+\Z_pp^{-\mu_{a+r_++1}}f_{a+r_++1}+\dots+\Z_pp^{\mu_n}e_n+\Z_pp^{-\mu_n}f_n\\&
\\& =\Z_p\tilde{e}_1+\Z_p\tilde{f}_1+\dots+\Z_p\tilde{e}_n+\Z_p\tilde{f}_n,
  \end{split}
\end{equation*}
with $0\le r_- \le \min(a,b'), 0\le r_+\le \min(a',b)$,\\ $b'=b-r_++r_-,a'=a-r_-+r_+$, \\$1\le\mu_1\dots\le  \mu_{r_-}$, $0\le
\mu_{r_-+1}\dots\le\mu_a$,\\ $0\le\mu_{a+1}\le\dots\le\mu_{a+r_+}$,
$0\le\mu_{a+r_++1}\le\dots\le\mu_n$.
\begin{enumerate}
  \item   Let $L,L'$ be
    $p$-elementary lattices  on $V$ which are
 isometric.\\
 Then $L'$ and $L$ are in the same
 $Sp(\Lambda)$-orbit if and only if they have the same elementary
 divisors with  respect to both $\Lambda$ and the dual lattice
 $\Lambda^\#$ of $\Lambda$.
\item
The orbits in a) of $L$ with $0\le a'\le n$ unimodular hyperbolic
planes in the canonical decomposition are
represented by the lattices
$ L(r_+,r_-,\mu_1,\ldots,\mu_n)$ with $r_+,r_-,\mu_1,\ldots,\mu_n$ as above.
\end{enumerate}
\end{proposition}
\begin{proof}
 For a) we denote by $a',b'$ the (common) numbers of unimodular respectively
 $p$-modular hyperbolic planes in the canonical decompositions of $L$
 and $L'$.
 It is well known (see \cite{garrettbook,frisch,abramenko_nebe,rsp_hyperbolic}) that there exist a
 basis $e_1,\ldots,e_n$, $f_1,\ldots,f_n$ of $\Lambda$ with $\langle
 e_i,e_j\rangle =\langle f_i,f_j\rangle =0$, $\langle e_i,f_j \rangle
 =\delta_{ij}$ for $1\le i \le a$, $\langle e_i,f_j\rangle
 =p\delta_{ij}$ for $a+1 \le i \le n=a+b$ and non negative integers $\mu_1, \ldots
 \mu_n,\nu_1,\ldots,\nu_n$ such that the $p^{\mu_i}e_i,p^{-\nu_i}f_i$ form a
 basis of $L$. 
Since $\Lambda$ and $L$ are both assumed to be $p$-elementary, we
must obviously have $\nu_i\in \{\mu_i,\mu_i-1\}$ 
for $1 \le i \le a$ and
$\nu_i \in \{\mu_i,\mu_i+1\}$ for $a<i\le n$; moreover, if $r_+$ is
the number of 
of indices $i$ with $\nu_i=\mu_i+1$ and $r_-$  the number of
indices $i$ with $\nu_i=\mu_i-1$,  we have  $b'=b-r_++r_-$ and $a'=a-r_-+r_+$.

In other words: There are precisely
$r_-$ among the subspaces $\Q_pe_i+\Q_pf_i$ for which the intersection
with $\Lambda$ is unimodular and the intersection with $L$ is
$p$-modular and for $r_+$ such subspaces the roles of
$\Lambda, L$ are reversed. We say that the $e_i,f_i$ are a
para-symplectic elementary divisor basis for the pair $\Lambda, L$
with exponent pairs  $(\mu_i,\nu_i)$.%

It is clear that $L,L'$ in the same $Sp(\Lambda)$-orbit
have the same elementary divisors with respect to $\Lambda$ and with
respect to $\Lambda^\#$. On the other hand, assume that 
there are para-symplectic elementary divisor bases
$\{e_i,f_i\}$ with exponent pairs $(\mu_i, \nu_i)$ as above for
$\Lambda,L$ and $\{e_i',f_i'\}$ with the same exponent pairs  $(\mu_i'=\mu_i,
\nu_i'=\nu_i)$ for  $\Lambda,L'$.

One has then a $\phi \in Sp(\Lambda)$ with
$\phi(e_i)=e_i',\phi(f_i)=f_i'$ for $1\le i\le n$, and this $\phi$
maps $L$ onto $L'$, i.e, $L$ and $L'$ are in the same orbit.

\medskip
We are therefore left with the task of proving that the elementary
divisor condition of the assertion implies that if one has 
para-symplectic elementary divisor bases 
$\{e_i,f_i\}$ with exponent pairs $(\mu_i, \nu_i)$ for
$\Lambda,L$ and $\{e_i',f_i'\}$ with exponent pairs $(\mu_i',\nu_i')$
for $\Lambda,L'$, those pairs must be the same $(\mu_i'=\mu_i,
\nu_i'=\nu_i)$.
Notice that this will imply in particular that two
elementary divisor bases for a fixed
$L$ must have the same exponent pairs $(\mu_i, \nu_i)$.

Let $\mu$ be maximal such that the number of  pairs $(\mu_i, \nu_i)$
with $\mu_i=\mu, \nu_i=\nu$ is not equal to the number of  pairs $(\mu_i', \nu_i')$
with $\mu_i'=\mu, \nu_i'=\nu$ for one of the possible values $\nu=\mu-1,
\nu=\mu, \nu=\mu+1$ of $\nu_i, \nu_i'$. The pairs with $\nu_i=\mu-1$
contribute elementary divisors $p^{\mu},p^{-\mu+1}$ of $L$ in $\Lambda^\#$,
the pairs with $\nu_i=\mu$ contribute elementary divisors $p^\mu,
p^{-\mu}$ or $p^{\mu+1},p^{-\mu+1}$ of $L$ in $\Lambda^\#$, the pairs
with $\nu_i=\mu+1$ contribute elementary divisors $p^{\mu+1},p^{-\mu}$
of $L$ in $\Lambda^\#$, and analogously for the $\nu_i'$ and $L'$.

Since here the only contributions to elementary divisors $p^{-\mu-1}$
in $\Lambda$ come  
from pairs $(\mu,-\mu-1)$ the number of such pairs $(\mu_i,\nu_i)$ and
$(\mu_i',\nu_i')$ must be equal. These pairs contribute elementary
divisors $p^{\mu+1}$ in $\Lambda^\#$, with the only other contributions
coming from pairs $(\mu,-\mu)$ and pairs with higher first entry. The
latter ones occur in equal numbers for $L,L'$ by assumption, so the
number of pairs $(\mu_i=\mu,\nu_i=-\mu)$ and
$(\mu_i'=\mu,\nu_i'=-\mu)$ with elementary divisors
$p^{\mu+1},p^{-\mu+1}$ must also be equal.

We are then left with
the possibility that the number s of pairs $(\mu,-\mu)$ with
elementary divisors $p^{\mu},p^{-\mu}$ in $\Lambda^\#$ are different
for $L$ and $L'$, say there are more such pairs for $L$ than for
$L'$. Consequently 
we have at this level fewer elementary divisors
$p^{-\mu}$ in $\Lambda^\#$ for $L'$ than for $L$. But by our
considerations above levels $\mu-1$ and lower can not contribute such
elementary divisors, so we have a contradiction to the assumption that
$L$ and $L'$ have the same elementary divisors in $\Lambda$ and in
$\Lambda^\#$, and we have finished the proof of a).

Moreover, the proof of a) shows that the lattices $L(r_+,r_-,\mu_1,\ldots,\mu_n)$ in
b) are
indeed a set of representatives of the orbits. 
\end{proof}
\begin{theorem}\label{hecke_representatives_local}
With notations as in the proposition denote by $\Gamma_{a,b}$ the paramodular group ${\Gamma}(
 \bigl( \begin{smallmatrix}
   1_a&\\
&p1_b
  \end{smallmatrix}\bigr))\subseteq Sp_n(\Q_p)$ over the $p$-adic integers $\Z_p$, where $a$ is the number of
  unimodular hyperbolic planes in $\Lambda$ and $b$ the number of
  $p$-modular hyperbolic planes.

Then the double cosets $\Gamma_{a,b} \alpha \Gamma_{a',b'}$ with
$\alpha \in Sp_n(\Q_p)$ 
have as a set of representatives the block diagonal matrices
\begin{equation*}
D(r_+,r_-,\mu_1,\ldots,\mu_n)=
\begin{pmatrix}
  B(r_+,r_-,\mu_1,
  \ldots,\mu_n)&0_n\\
  0_n&{}^t(B(r_+,r_-,\mu_1,
\ldots,\mu_n))^{-1}
\end{pmatrix}
\end{equation*}

with $r_-,r_+$ and the $\mu_i$ as in the above proposition, where $B(r_+,r_-,\mu_1,
  \ldots,\mu_n)$ equals
%
  \begin{equation*}
  \begin{pmatrix}
    &&&&&&p^{\mu_1}&&&&&\\
&&&&&&&\ddots&&&&\\
&&&&&&&&&p^{\mu_{r_-}}&&&\\
&&&p^{\mu_{r_-+1}}&&&&&&&&&\\
&&&&\ddots&&&&&&&&\\
&&&&&p^{\mu_a}&&&&&&&\\
p^{\mu_{a+1}}&&&&&&&&&&&\\
&\ddots&&&&&&&&&&\\
&&p^{\mu_{a+r_+}}&&&&&&&&&\\
&&&&&&&&&&p^{\mu_{a+r_++1}}&&\\
&&&&&&&&&&&\ddots&\\
&&&&&&&&&&&&p^{\mu_n}
  \end{pmatrix}
\end{equation*}
with the rest of the entries being zero.
\end{theorem}
\begin{proof} We let $\Lambda, L$ and the basis $(e_1,\ldots, e_n,
  f_1,\ldots,f_n)$ of $\Lambda$ be as in the proof of the proposition
  and obtain  a standard symplectic basis $\{x_1=e_1,\ldots,x_n=e_n$, $y_1=p^{-1}f_1,\ldots,y_n=p^{-1}f_n\}$ of $V$.
Then with respect to this standard basis $Sp(\Lambda)$ has $\Gamma_{a,b}$
and $Sp(L)$ has $\Gamma_{a',b'}$ as its
group of matrices. The double cosets
$\Gamma_{a,b}\alpha\Gamma_{a',b'}$ correspond then bijectively to the
$Sp(\Lambda)$-orbits of lattices isometric to $L$ and are represented
by the matrices $  D(r_+,r_-,\mu_1,\ldots,\mu_n) $ which transform $L$ into the standard lattices given in
part b) of the proposition. 
\end{proof}
\begin{lemma}\label{transpose_integrality}
With notation as above and   $T:=\biggl(
\begin{smallmatrix}
1_{a'}&\\&p1_{b'}  
\end{smallmatrix}\biggr)$, 
$T':=\biggl(
\begin{smallmatrix}
 1_{a}&\\&p1_{b} 
\end{smallmatrix}\biggr)$ the matrix
\begin{equation*} T^{-1}\,{}^t\!B(r_+,r_,\mu_1,\dots,\mu_n) T'
  \end{equation*}
  has integral entries.

  In the case $a=a', b=b'$ we have with $r=r_-=r_+$
  \begin{align*}
    T^{-1}\,{}^t\!B(r,r,\mu_1,\dots,\mu_n) T %
        =B(r,r,&\mu_{a+1}+1,\dots,\mu_{a+r}+1,
                     \mu_{r+1},\dots,\mu_a,\\&\mu_1-1,\dots,\mu_r-1,\mu_{a+r+1},\dots,\mu_n).
  \end{align*}
\end{lemma}
\begin{proof}
  This is easily checked, using $1\le \mu_1\le \dots\le\mu_{r_-}$.
\end{proof}
\begin{corollary}
With notations as above assume $a=a',b=b',r:=r_-=r_+$, and write
$\Gamma=\Gamma_{a,b}$, let ${\mathcal H}(\Gamma)$ denote the Hecke
algebra generated by the $\Gamma$ double cosets as defined in
\cite{shimurabook}. For $j \in \N$ denote by $T(p^j)$ the sum of the
double cosets of the $D(r,r,\mu_1,\ldots,\mu_n)$ with $\mu_1+\dots
+\mu_n=j$. Then $T(p^j)$ is invariant under the involution of ${\mathcal
  H}(\Gamma)$ given by $\Gamma \alpha \Gamma \mapsto \Gamma
\alpha^{-1} \Gamma$.

Consequently, the subalgebra generated by the $T(p^j)$ is commutative.
\end{corollary}
\begin{proof}
  If we conjugate the   inverses
  $D(r=r_+,r=r_-,\mu_1,\ldots,\mu_n)^{-1}$ of the representatives
  given above of the double cosets occurring
  in $T(p^j)$ by
  \begin{equation*}
    \begin{pmatrix}
      &&1_a&\\
      &&&p^{-1}1_b\\
      -1_a&&&\\
      &-p1_b&&
    \end{pmatrix}\in \Gamma_{a,b}
  \end{equation*} we obtain by the previous lemma a permutation of the representatives of
  the double cosets in $T(p^j)$. %

The asserted commutativity
can then be seen  from an easy modification of the proof  of
\cite[Prop. 3.8]{shimurabook}.
\end{proof}
\begin{remark}
For later use we notice that $T(p^0)$ consists  only of the double coset
  with $r=\mu_1=\ldots=\mu_n=0$.
\end{remark}

In the case of level $1$, i.e., the usual full integral symplectic
group, it is well known that the formal power series $\sum_{j=0}^\infty
T(p^j)X^j$ is a rational function in $X$, and that this rational function
can be  explicitly given (essentially equal to the local Euler-factor of the 
standard $L$-function),  
see
\cite{
boecherer_rationalitaetssatz}.
We have unfortunately not succeeded  in finding  a similar result in the
paramodular case. For the applications we have in mind we can,
however, at least partly substitute this result by an upper estimate for the
number of one sided cosets that occur when one expands each of the
double cosets occurring in $T(p^j)$ into a formal sum of single
cosets. For this we need a few preparations. For  $j\ge 0$ we denote by $N(p^j)$ the number of left cosets $\Gamma
\gamma$ occurring in the decomposition of the double cosets in
$T(p^j)$ into left cosets. In analogy to the case of symmetric
bilinear forms, see \cite{kneser_nachbarn}, we say that isomorphic
lattices $\Lambda, \Lambda'=\alpha \Lambda \quad \alpha \in Sp(V)$ on
$V$ are neighbors of each other if $\Lambda \cap \Lambda'$ has index
$p$ in both $\Lambda$ and $\Lambda'$.
\begin{lemma}
Let $j\ge 0$.
\begin{enumerate}
\item 
The cosets $\alpha Sp(\Lambda)$ with $\alpha \in Sp(V)$ correspond
bijectively to the lattices $\alpha \Lambda$ in $V$.
\item The double coset $Sp(\Lambda) \alpha Sp(\Lambda)$ appears in
  $T(p^j)$ if and only if $\Lambda \cap \alpha \Lambda$ has index
  $p^j$ in $\Lambda$ and in $\alpha \Lambda$. 
  \item $N(p)$ is the number of lattices in $V$ which are neighbors of
    $\Lambda$.
    \item $N(p^j)\le (N(p))^j$.
\end{enumerate}
\end{lemma}
\begin{proof}
 a) is obvious, b) follows from the definition of $T(p^j)$. For c) and
 d) we notice that the number of left cosets in $\Gamma \alpha \Gamma$
 equals the number of right cosets. c) is then immediate from
 a),b). Finally, as noticed in the proof of Theorem
 \ref{hecke_representatives_local}, isomorphic lattices $\Lambda,
 \Lambda'=\alpha\Lambda$ of level dividing $p$ have adapted parasymplectic bases
 $(e_1,\dots,e_n,f_1,\dots,f_n)$, $(p^{\mu_1}e_1,\dots, p^{\mu_n}e_n$,
 $p^{\nu_1}f_1,\dots,p^{\nu_n}f_n)$. From this we see that the
 lattices $\Lambda,
 \Lambda'=\alpha\Lambda$ with $(\Lambda:\Lambda \cap
 \alpha\Lambda)=p^j$ can be  connected by a chain of neighboring
 lattices of length $j$, which implies d). Notice, however, that
 $N(p)^j$ will usually be larger than $N(p^j)$, since chains of
 neighboring lattices may backtrack and since different chains can
 arrive at the same goal.
\end{proof}
\begin{lemma}
Let $n=n_1+n_2$ and assume that $\Lambda$ is the orthogonal sum of
$n_1$ unimodular hyperbolic planes and $n_2$ $p$-modular hyperbolic
planes.

Then the number $N(p)$ of neighbors of $\Lambda$ equals
\begin{equation*}
  \frac{p}{(p-1)^2 }(p^{2n_1}-1)(p^{2n_2}-1)
+\frac{p}{p-1}p^{2n_1}(p^{2n_2}-1)
+\frac{p}{p-1}p^{2n_2}(p^{2n_1}-1).
\end{equation*}
If $p>3$ or $n_2=0$ holds one has $N(p)< p^{2n+1}$, for $p=3$
we have $N(3)< \frac{5}{4}3^{2n+1}$, for $p=2$ one has $N(2)<2^{2n+2}$.
\end{lemma}
\begin{proof}
There are three different types of neighbors $\Lambda'$ of  a given $\Lambda$
which we will count separately.

In the first type there is $x\in \Lambda$ with $\frac{x}{p}\in
\Lambda'$ and $\langle x, \Lambda\rangle=\langle \frac{x}{p}, \Lambda'
\rangle=\Z_p$. We have then
$\Lambda'=:\Lambda(x)=\Z_p \frac{x}{p}+\Lambda_x$ with $\Lambda_x:=
\{z \in \Lambda \mid \langle z,x\rangle \in  p\Z_p\}$. If we have two
vectors $x,x'$ as above with $\Lambda(x)=\Lambda(x')$ and $z' \in
\Lambda_{x'}$, we can write $pz'=\alpha x + pz$ with $\alpha \in \Z_p,
z \in \Lambda_x$ and obtain $\langle pz',x\rangle \in p^2\Z_p$, thus
$\Lambda_{x'} \subseteq \Lambda_x$ and hence
$\Lambda_x=\Lambda_{x'}$. On the other hand, if we have $x,x'\in
\Lambda$ as above with 
$\Lambda_x=\Lambda_{x'}$ we have $\Lambda(x) =\Lambda(x')$ if
and only if $x,x'$ generate the same line modulo
$p\Lambda_x=p\Lambda_{x'}\supseteq p^2\Lambda$. Obviously, one has
$\Lambda(x)=\Lambda(x')$ if $x,x'$ generate the same line modulo $p^2\Lambda$.
The number of eligible lines $\Z_px+p^2\Lambda$ is
\begin{equation*}
  \frac{(p^{2n}-p^{2n_2})p^{2n}}{p^2-p}=\frac{p^{2n+2n_2}(p^{2n_1}-1)}{p^2-p}.
\end{equation*}
Writing $\Lambda=(\Z_px+\Z_p\tilde{x})\perp \tilde{\Lambda}$ with
  $\langle x, \tilde{x}\rangle=1$ we have
  $\Lambda_x=(\Z_px+p\Z_p\tilde{x})\perp \tilde{\Lambda}$ and see that
    each line modulo $p\Lambda_x$ contains $p^{2n}$ points modulo
    $p^2\Lambda$, whereas each line modulo $p^2\Lambda$ consists of
    $p^2$ points. 
Each line modulo $p\Lambda_x$ consists therefore of $p^{2n-2}$ lines
modulo $p^2 \Lambda$  yielding the same neighbor of $\Lambda$, and
we obtain the last summand  in our formula for $N(p)$ as the number of
neighbors of the first type.

In the second type there is  a primitive vector $y\in \Lambda$ with
$\frac{y}{p} \in \Lambda'$ and $\langle y,\Lambda\rangle=\langle
\frac{y}{p}, \Lambda'\rangle =p\Z_p$, one
has then $\Lambda'=:\Lambda(y)=\Z_p \frac{y}{p}+\Lambda_y$ with
$\Lambda_y:=\{z\in \Lambda\mid \langle z,y\rangle \in p^2\Z_p\}$. In
the same way as above we see that $\Lambda(y)=\Lambda(y')$ implies
$\Lambda_y=\Lambda_{y'}$ and that vectors $y,y'$ as above with
$\Lambda_y=\Lambda_{y'}$ yield the same neighbor if and only if 
they generate the same line modulo $p\Lambda_y=p\Lambda_{y'}\supseteq
p^3\Lambda$.
Obviously, one has $\Lambda(y)=\Lambda(y')$ if $y,y'$ generate the
same line modulo $p^3\Lambda$.
The number of eligible lines modulo $p^3\Lambda$ is 
\begin{equation*}
  \frac{p^{4n+2n_1}(p^{2n_2}-1)}{p^3-p^2}.
\end{equation*}
Writing $\Lambda=(\Z_py+\Z_p\tilde{y})\perp \tilde{\Lambda}$ with
$\langle y,\tilde{y} \rangle =p$ we have
$\Lambda_y=(\Z_py+p\Z_p\tilde{y})\perp \tilde{\Lambda}$ and see that
the line $\Z_py+p\Lambda_y$ consists of $p^{4n}$ points modulo
$p^3\Lambda$, whereas each line modulo $p^3 \Lambda$ consists of $p^3$
points.
Each line $\Z_py+\Lambda_y$ consists therefore of $p^{4n-3}$ lines
modulo $p^3\Lambda$ all yielding the same neighbor of $\Lambda$, and
we obtain the second summand in our formula for $N(p)$ as the number
of neighbors of the second type.

In the third type there is a vector $x \in \Lambda\cap\Lambda', x\not \in p\Lambda'$ with $\langle x,
\Lambda\rangle =\Z_p, \langle x, \Lambda' \rangle =p\Z_p$ and a
primitive vector $y \in \Lambda$ with $\frac{y}{p}\in \Lambda', \langle 
y,\Lambda\rangle=p\Z_p, \langle \frac{y}{p}, \Lambda'\rangle =\Z_p$,
we have then $\Lambda'=:\Lambda(x,y)=\Z_p\frac{y}{p}+\Lambda_x$ with
$\Lambda_x=\{z\in \Lambda \mid \langle z,x \rangle\in p\Z_p$. As in
the previous cases we see that $\Lambda(x,y)=\Lambda(x',y')$ implies
$\Lambda_x=\Lambda_{x'}$ and that for given $\Lambda_x$ two vectors  $y,y'$ as above yield the
same neighbor if and only if they generate the same line modulo
$p\Lambda_x\supseteq p^2\Lambda$.

Obviously one has
$\Lambda(x,y)=\Lambda(x',y')$ if $\Lambda_x=\Lambda_{x'}$ and $y,y'$
generate the same line modulo $p^2\Lambda$. We have
$\Lambda_x=\Lambda_{x'}$ if and only if $x,x'$ generate the same line
modulo $\Lambda \cap p\Lambda^\#$, and one has $(\Lambda:\Lambda\cap
p\Lambda^\#)=p^{2n_1}$, the number of possible $\Lambda_x$ is hence
equal to $\frac{p^{2n_1}}{p-1}$. For a fixed $\Lambda_x$ we may write
$\Lambda_x+\Z_p\frac{y}{p}=(\Z_px+p\Z_p\tilde{x})\perp
(\Z_p\frac{y}{p}+Z_p\tilde{y})\perp \tilde{\Lambda}$ with $\langle x,
\tilde{x}\rangle=1, \langle y, \tilde{y}\rangle =p$ by changing $x$
modulo $\Lambda \cap p\Lambda^\#$ if necessary. We see that
$\Z_py+p\Lambda_x$ consists of $p^{2n}$ points modulo
$p^2\Lambda$. Since a line modulo $p^2\Lambda$ consists of $p^2$
points, the line $\Z_py+p\Lambda_x$  consists of $p^{2n-2}$ lines
modulo $p^2\Lambda$ all yielding, together with $\Lambda_x$, the same
neighbor. The number of eligible lines modulo $p^2\Lambda$ is
\begin{equation*}
  \frac{p^{2n}(p^{2n_2}-1)}{p^2-p},
\end{equation*}
dividing by $p^{2n-2}$ and multiplying by the number of $\Lambda_x$ we
obtain the first summand in our formula for $N(p)$ as the contribution
of the neighbors of the third type.

Finally, the upper bounds for $N(p)$ are easily checked using our formula.
\end{proof}
For use in later sections we have to make the transition from our
local considerations to the global setting more explicit. We therefore
switch now to ground field $\Q$ and a lattice $\Lambda$ of square free
level $N$ on the $\Q$-vector space $V$, put $\Gamma=Sp(\Lambda)$ and $\Gamma_p=Sp(\Lambda_p)$.
As mentioned in the beginning of this section, the Hecke algebra
${\mathcal H}(Sp_n(V), \Gamma)$ is isomorphic to the restricted tensor
product of the local Hecke algebras ${\mathcal H}(Sp_n(V_p), \Gamma_p)$.

\begin{corollary}\label{blockdiagonal_rep}
  Let $ \Gamma=\Gamma^{(n)}(T), \Gamma'=\Gamma^{(n)}(T')$  with
$T=\diag(d_1,\dots,d_n)$, $T'=\diag(d_1',\dots,d_n')$, $d_i\mid d_{i+1},
  d'_i\mid d'_{i+1}$
  be paramodular subgroups of $Sp_n(\Q)$ of
square free levels $N,N'$.
Then each double coset $\Gamma' \alpha \Gamma$ with $\alpha \in
Sp_n(\Q)$ has a representative of the form
\begin{equation*}
  \begin{pmatrix}
    B&0\\0&{}^tB^{-1}
  \end{pmatrix}, 
\end{equation*}
with $B\in GL_n(\Q)\cap M_n(\Z)$ satisfying $T^{-1}{}^t\!BT' \in M_n(\Z)$.
\end{corollary}
\begin{proof}
We write $\Gamma=Sp_n(\Q)\cap \prod_p\Gamma_p$  with local paramodular
groups $\Gamma_p\subseteq Sp_n(\Q_p)$ and $\Gamma_p=Sp_n(\Z_p)$ for
all primes $p \nmid N$, and similarly for $\Gamma'$. We choose local
representatives $\alpha_p$ of the double cosets $\Gamma_p \alpha
\Gamma_p'$ as in  Theorem \ref{hecke_representatives_local}, in
particular, each $\alpha_p$ is of  block diagonal shape 
$\bigl(
\begin{smallmatrix}
  A_p&0\\0&{}^t\!A_p^{-1}
\end{smallmatrix}\bigr)$, where in $A_p\in M_n(\Z_p)$ each line has a single
nonzero entry. Let $M\in \N$ with $N\mid M, N'\mid M$ be such that
$\alpha_p\in \Gamma_p=\Gamma'_p=Sp_n(\Z_p)$ for all $p\nmid M$.
Changing the local representatives  $\alpha_p$ if necessary we can
assume that there exists $m \in \N$ with $\det(A_p)=m$ for all $p\mid
M$.
We can hence find a diagonal matrix $D\in GL_n(\Q)$ with $DA_p\in
SL_n(\Z_p)$ for all primes $p\mid M$ and $D^{-1}\in M_n(\Z)$. By the strong
approximation theorem for $SL_n$ there exists $A \in SL_n(\Z)$ with
$A\equiv DA_p \bmod MM_n(\Z_p)$ for all primes $p\mid M$, and
$B=D^{-1}A$ is as asserted, where for the last part of the assertion
we use Lemma \ref{transpose_integrality}; notice that we have reversed
the roles of $T,T'$ here in view of the application of this result in the next section. 
\end{proof}
\begin{lemma}
In the Hecke algebra $ {\mathcal H}(Sp_n(V),\Gamma)$ denote for $m\in
\N$ by $T(m)$ the sum of all double cosets with a representative of the form
\begin{equation*}
  \begin{pmatrix}
    B&0\\0&{}^tB^{-1}
  \end{pmatrix}, 
\end{equation*}
as in Corollary \ref{blockdiagonal_rep} with $B\in GL_n(\Q)\cap M_n(\Z)$ of determinant $m$.
Then $T(m)$ is the product of the $T(p_i^{t_i})$, where
$m=\prod_ip_i^{t_i}$ is the factorization of $m$ into prime powers.
\end{lemma}
\begin{proof}
 If $\Gamma \alpha \Gamma$ is a double coset occurring in $T(m)$ as a
 summand and $p=p_i$ for some $i$, the coset $\Gamma_p \alpha
 \Gamma_p$ occurs as a summand in $T_p(p_i^{t_i})$, where we write
 $T_p(\dots)$ for the elements of the local Hecke algebra at $p$.
 Moreover, the isomorphism between $ {\mathcal H}(Sp_n(V),\Gamma)$ and
 the restricted tensor product of the $ {\mathcal
   H_p}(Sp_n(V_p),\Gamma_p)$ guarantees, together with the previous
 corollary, that each collection of local double cosets
 $\Gamma_p\alpha_p\Gamma_p$ with $\alpha_p\in\Gamma_p$ for almost all
 $p$ arises in this way from a global double coset occurring in some
 $T(m)$. Consequently, writing each $T(p_i^{t_i})$ as a sum of double 
 cosets and distributively multiplying out their product we obtain the
 sum of all double cosets occurring in $T(m)$.  
\end{proof}

\section{Garrett's double coset decomposition for paramodular groups}\label{sec4}
Our goal in this section is to generalize Garrett's double coset
decomposition from \cite{garrett} from the integral symplectic group
to paramodular groups. For this we want to investigate  orbits of
maximal totally isotropic submodules of lattices with nondegenerate
alternating bilinear form.

An obvious consequence of Theorem \ref{isotropic_submodule_bases}  and
Corollary \ref{orbits_isotropic} is:
\begin{corollary}\label{orbits_submodules}
  With $\Lambda$ as in  Theorem \ref{isotropic_submodule_bases} the  group $Sp(\Lambda)$
acts transitively on the set of maximal totally 
isotropic submodules of $\Lambda$. 

More generally, let $X_1,X_2\subseteq \Lambda$ be  sublattices and $Z_1,Z_2\subseteq
X_1,X_2$  primitive totally isotropic submodules of 
$\Lambda$ of rank $r$ with $Z_j\subseteq \rad(X_j)$ for $j=1,2$. For $j=1,2$ let $M_j=M(Z_j)$ as in  Theorem
\ref{isotropic_submodule_bases} 
with $\Lambda=M(Z_j)\perp
\Lambda'(Z_j)=M_j\perp \Lambda_j'$,
$X_j=(X_j\cap \Lambda'_j)\perp Z_j$, put $X'_j=X_j\cap \Lambda'_j$.   

Then there exists
$\phi \in Sp(\Lambda)$ with $\phi(X_1)=X_2, \phi(Z_1)=Z_2$ if and only
if  $d(Z_2)\in
d(Z_1)R^\times$ holds and there exists an isometry $\psi:\Lambda_1'\to
\Lambda_2'$ with $\psi(X_1')=X_2'$. 

A set of representatives of the orbits of pairs $X\supseteq Z$ as
above with $Z=\rad(X)$  is then obtained by picking a  $Z=Z_d$ for
each $dR^\times$ satisfying the conditions of Corollary
\ref{orbits_isotropic}  and setting $X=X'\perp Z_d$, where $X'$ runs through
a set of representatives of the $Sp(\Lambda'(Z_d))$-orbits of nondegenerate sublattices of $\Lambda'(Z)$.  
\end{corollary} 
If $X\subseteq \Lambda$ is a fixed maximal totally isotropic
submodule and $P=P_X\subseteq Sp(\Lambda)$ its stabilizer, we may
therefore identify the set of all maximal totally isotropic submodules
with the coset  space $\{ gP\mid g\in Sp(\Lambda)\}$, and 
for an orthogonal splitting
$\Lambda=\Lambda_1 \perp \Lambda_2$
the orbits of maximal totally isotropic submodules of $\Lambda$ 
under the action of $Sp(\Lambda_1)\times
Sp(\Lambda_2)\subseteq Sp(\Lambda)$ correspond bijectively to
the double cosets $(Sp(\Lambda_1)\times
Sp(\Lambda_2)) gP$ with   $g\in Sp(\Lambda)$. Garrett gave explicit
representatives for these double cosets for $\Lambda$ of level $1$ in
\cite{garrett}. We will generalize his result to arbitrary square
free level by finding representatives of the orbits of maximal totally
isotropic submodules of $\Lambda$  
under the action of $Sp(\Lambda_1)\times
Sp(\Lambda_2)\subseteq Sp(\Lambda)$.
\begin{theorem}\label{xdecomp}
  Let $\Lambda$ be a free module over  the %
  commutative ring $R$ with an
   alternating bilinear form $\langle,\rangle$.

  Let $\Lambda_1,\Lambda_2$ be mutually orthogonal (with respect to
  $\langle, \rangle$) submodules of
  $\Lambda$ with %
  $\Lambda=\Lambda_1 \perp \Lambda_2$ and denote by $\pi_1,\pi_2$ the orthogonal 
projections from $\Lambda$ onto $\Lambda_1,\Lambda_2$.

Let $X$ be a  totally isotropic submodule of $\Lambda$ and
write $X_i=\pi_i(X)$ for $i=1,2$. %

Then there exists an isometry $\phi$ from the bilinear module
$(X_1/\rad(X_1), \langle, \rangle)$ to 
$(X_2/\rad(X_2),-\langle, \rangle)$ such that
\begin{equation*}
  X=\{x_1+x_2 \mid x_i \in X_i, x_2+\rad(X_2)=\phi(x_1+\rad(X_1))\}.
\end{equation*}
If $X$ is maximal totally isotropic  the triple $(X_1,X_2,\phi)$ is maximal
with respect to this property.

Conversely, 
for $i=1,2$ let $Y_i\subseteq \Lambda_i$ be
submodules with an isometry
$\psi:(Y_1,\langle,\rangle)/\rad(Y_1)\to (Y_2,-\langle,\rangle)/\rad(Y_2)$

Then 
\begin{equation*}
Y:=\{ (y_1,y_2) \mid y_1\in Y_1, y_2 \in Y_2, \psi(y_1+\rad(Y_1))=y_2+\rad(Y_2)\}  
\end{equation*}
is a totally isotropic submodule of $\Lambda$ with
$\pi_1(Y)=Y_1,\pi_2(Y)=Y_2$.
If $Y_1,Y_2,\psi$ are maximal, $Y$ is  a maximal totally isotropic
submodule of $\Lambda$.
\end{theorem}
\begin{proof}
Let $v_1=\pi_1(x) \in X_1$ with $x=v_1+v_2 \in X, v_2 \in X_2$. If
$v_1=\pi_1(v_1+v_2')$ for some $v_2'\in X_2$ and $w_2=\pi_2(y) \in X_2$
is arbitrary, we have 
\begin{equation*}
  \langle v_2-v_2',w_2\rangle = \langle v_2-v_2', y\rangle=0,
\end{equation*}
and hence $v_2+\rad(X_2)=v_2'+\rad(X_2)$
since $v_2-v_2'=(v_1+v_2)-(v_1+v_2')\in X$ and $X$ is totally
isotropic.  

Moreover, if $v_2 \in \rad(X_2)$ holds and $w_1=\pi_1(y)\in X_1$ is
arbitrary with $y=w_1+w_2, w_2\in X_2$ we have
\begin{eqnarray*}
  \langle w_1,v_1\rangle&=&\langle w_1,v_1+v_2\rangle\\
&=&\langle w_1+w_2, v_1+v_2 \rangle\\
&=&\langle y,x\rangle\\
&=&0,
\end{eqnarray*}
since $v_2 \in \rad(X_2)$ gives $\langle w_2,v_2\rangle=0$ and since
$X$ is totally isotropic.
Conversely, it is easy to see that  a vector $v_1\in \rad(X_1)$  gives
a vector $v_2\in \rad(X_2)$.

By the homomorphism theorem we obtain hence a map $\phi:X_1/\rad(X_1)
\to X_2/\rad(X_2)$ with $\phi(v_1)+\rad(X_1)=v_2+\rad(X_2)$ for all
$v_1+v_2 \in X$ with $v_i \in \Lambda_i$ for $i=1,2$.

For $v_1=\pi_1(v_1+v_2), v_1'=\pi_1(v_1'+v_2')$ we have
\begin{equation*}
 \langle v_1,v_1'\rangle +\langle v_2,v_2'\rangle=\langle
 v_1+v_2,v_1'+v_2'\rangle =0 
\end{equation*}
since $X$ is totally isotropic, so $\phi$ is indeed an isometry.

If on the other hand $Y_1,Y_2,\psi, Y$ are as in the assertion it is clear that $Y$ is a
totally isotropic submodule of $\Lambda$ with
$\pi_1(Y)=Y_1,\pi_2(Y)=Y_2$. If $Y_i\supseteq X_i$ for $i=1,2$ and
$\psi$ extends $\phi$ one has $Y\supseteq X$, and one sees that the
maximality of $X$ implies the maximality of $X_1,X_2, \phi$. If $w \in \Lambda$ is such that $Y+Rw$
is totally isotropic, $\tilde{Y_1}:=Y_1+R\pi_1(w),
\tilde{Y_2}:=Y_2+R\pi_2(w)$ have the same properties as $Y_1,Y_2$, so
by the assumed maximality we have $\pi_1(w)\in Y_1, \pi_2(w)\in Y_2$,
which implies $w \in Y$. Hence the maximality of $Y_1,Y_2,\psi$
implies that $Y$ is indeed a maximal totally
isotropic submodule of $\Lambda$. 
\end{proof}
\begin{remark}
  In the situation of the theorem with $X$ maximal let $R$ be an
  integral domain with 
  field of fractions $F$ and let $\Lambda, \Lambda_1, \Lambda_2$ be
  lattices of full rank on the $F$-vector spaces $V,V_1,V_2$. Applying
  the theorem to $FX\subseteq V, V_1,V_2$  we obtain subspaces
  $U_i=\pi_i(FX)\subseteq V_i$ and an isometry $\phi:U_1/\rad(U_1)\to
  U_2/\rad(U_2)$.

  We have then $X_1=(U_1\cap \Lambda_1)\cap
  \phi^{-1}(U_2\cap \Lambda_2)$ and $X_2=(U_2 \cap
  \Lambda_2)\cap\phi(U_1\cap \Lambda_1)$ and may view the map
  $X_1/\rad(X_1)\to X_2/\rad(X_2)$ as the restriction of the map
  $U_1/\rad(U_1)\to U_2/\rad(U_2)$ (with the natural embedding of
  $X_i/\rad(X_i)$ into $U_i/\rad(U_i)$).

  Conversely, given subspaces
  $U_i\subseteq V_i$ and an isometric map $\phi$ as above we can set $X_1=(U_1\cap \Lambda_1)\cap
  \phi^{-1}(U_2\cap \Lambda_2)$ and $X_2=(U_2 \cap
  \Lambda_2)\cap\phi(U_1\cap \Lambda_1)$ and retrieve the associated
  maximal totally isotropic submodule $X$ of $\Lambda$ as described in the theorem.
\end{remark}
\begin{lemma}\label{ranks}
  With notations as before let now $R$ be a principal ideal domain and
 $\Lambda_1,\Lambda_2$ finitely generated free $R$-modules, let $2m:=\rk(\Lambda_1),
 2n:=\rk(\Lambda_2)$ and assume that 
$\Lambda, \Lambda_1, \Lambda_2$ with the alternating form
$\langle,\rangle$ are nondegenerate alternating modules. Let $X$ be a
maximal totally isotropic submodule of $\Lambda$, put  
 $2r:=\rk(X_1/\rad(X_1))=\rk(X_2/\rad(X_2)$, $m_1=\rk(\rad(X_1)), n_1=\rk(\rad(X_2))$.

 Then $m_1=m-r, n_1=n-r$.
\end{lemma}
\begin{proof}
  From $X=\{x_1+x_2\mid x_1\in
X_1,x_2\in X_2, \phi(x_1+\rad(X_1))=x_2+\rad(X_2)\}$ with $\phi$ as in
the proof of the theorem one sees that
$m+n=\rk(X)=m_1+2r+n_1$ holds, and since $m$ resp. $n$ is the
dimension of a maximal totally isotropic submodule of $\Lambda_1$
resp.\ $\Lambda_2$ we have $m_1+r\le m, n_1+r \le n$. Taken together
we obtain $m_1=m-r, n_1=n-r$.
\end{proof}
\begin{lemma}\label{cancellationlemma}
 Let $\Lambda_1\subseteq V_1, \Lambda_2\subseteq V_2$ be
integral $\Z$-lattices of square free level $N$ in the regular symplectic spaces
$V_1,V_2$ over $\Q$, 
we set $V=V_1\perp V_2, \Lambda=\Lambda_1\perp
\Lambda_2$. Let $\pi_1,\pi_2$ denote the orthogonal projections from
$\Lambda$ onto $\Lambda_1, \Lambda_2$.

Let $X$ be a maximal totally isotropic submodule of $\Lambda$ and put
$X_i=\pi_i(X)$ for $i=1,2$, let $Z_i=X\cap \Lambda_i=\rad(X_i)$. Then
there are primitive sublattices $M_i\subseteq \Lambda_i, M_i\supseteq Z_i$ and
$Z_i$ maximal totally isotropic in $M_i$ with
$\Lambda_i=M_i\perp \Lambda_i'$.

With $M=M_1 \perp M_2, Z=Z_1\perp Z_2, \Lambda'=\Lambda_1'\perp
\Lambda_2', X'=X \cap \Lambda'$ one has $X=Z \perp X'$ and
$X_1=\pi_1(X')\perp Z_1, X_2=\pi_2(X')\perp Z_2$.

In particular, the $\pi_i(X')$ are nondegenerate isometric to
$X_i/\rad(X_i)$, and the isomorphism $\phi: X_1/\rad(X_1)\to
X_2/\rad(X_2)$ of Theorem \ref{xdecomp} induces an isomorphism
$\phi':\pi_1(X')\to \pi_2(X')$.
\end{lemma}
\begin{proof}
The lattices $M_i$ are obtained from Theorem
\ref{isotropic_submodule_bases}, which also implies that one
has $X=Z \perp X'$ and 
$X_1=\pi_1(X')\perp Z_1, X_2=\pi_2(X')\perp Z_2$ as asserted.
That $\phi$ induces $\phi'$ as asserted is clear.
\end{proof}
\begin{lemma}\label{orbit_char}
  In the situation of the previous lemma let (using the notation of
  Corollary \ref{orbits_isotropic})  $d:=d(Z_1),d'=d(Z_2)$.  Then the orbit of the submodule
  $X$ under the
  action of $Sp(\Lambda_1)\times Sp(\Lambda_2)\subseteq Sp(\Lambda)$
determines   $d,d'$ and the orbit of $X'$ under the action of
  $Sp(\Lambda_1')\times Sp(\Lambda_2')\subseteq Sp(\Lambda')$ and is
  conversely determined by these data.
\end{lemma}
\begin{proof}
This follows from the previous lemma and Corollary \ref{orbits_isotropic}.
\end{proof}

By these lemmata we can now restrict our attention to the case that
$\Lambda_1, \Lambda_2$ have equal rank and that the projections
$\pi_1(X), \pi_2(X)$ are nondegenerate alternating modules.

\medskip
\begin{proposition}\label{propo_fullrank}
 Let $\Lambda_1,\Lambda_2$ be $\Z$-lattices of ranks $2m$ on the
 vector spaces $V_1,V_2$ over $\Q$ with
 alternating bilinear forms of square free levels $N_1,N_2$ and
 determinants $D_1,D_2$, let
 $N=\lcm(N_1,N_2)$ and $D=D_1D_2$, let $\Lambda=\Lambda_1 \perp
 \Lambda_2$ and let $X,\hat{X}$ be  maximal 
 totally isotropic submodules of $\Lambda$ for which the projections
 $\pi_i(X)=X_i,\pi_i(\hat{X})=\hat{X}_i$ to $\Lambda_i$ are nondegenerate alternating modules
 for $i=1,2$.

 Let $\phi,\hat{\phi}:V_1\to V_2$ with $X_2=\Lambda_2\cap
 \phi(\Lambda_1)$, $X_1=\Lambda_1 \cap \phi^{-1}(\Lambda_2)$ and
  $\hat{X}_2=\Lambda_2\cap
 \hat{\phi}(\Lambda_1)$, $\hat{X}_1=\Lambda_1 \cap \hat{\phi}^{-1}(\Lambda_2)$ be the
 isometries associated to $\Q X, \Q\hat{X}$ 
 by Theorem \ref{xdecomp} and the remark following it.
 
 Let $W$ be a $2m$-dimensional vector space over $\Q$ with
 nondegenerate alternating bilinear form and $\Sigma_1:V_1 \to W, 
 \Sigma_2:V_2\to W$ be fixed isometries. Then
 \begin{enumerate}
 \item[a)]  For $\sigma_1\in Sp(V_1), \sigma_2 \in Sp(V_2)$ one
   has $\Q\hat{X}=(\sigma_1,\sigma_2)(\Q X)$ if and only if $\sigma_2
   \circ \phi\circ \sigma_1^{-1} =\hat{\phi}$ holds.
   \item[b)] $\hat{X}$ is in the $Sp(\Lambda_1)\times
     Sp(\Lambda_2)$-orbit of $X$ if and only if one has
     \begin{equation*}
       \Sigma_2 \circ \hat{\phi} \circ \Sigma_1^{-1}\in
       Sp(\Sigma_2\Lambda_2)(\Sigma_2\circ \phi \circ \Sigma_1^{-1}) Sp(\Sigma_1\Lambda_1).
     \end{equation*}
 \end{enumerate}
\end{proposition}
\begin{proof}
  Since one has $\hat{X}=(\sigma_1,\sigma_2)X$ if and only if
  $\hat{X}_1=\sigma_1X_1, \hat{X}_2=\sigma_2(X_2)$ by Theorem
  \ref{xdecomp}, a) follows. Assertion b) then follows from a).
\end{proof}
Taken together, Lemma \ref{orbit_char} and Proposition
\ref{propo_fullrank} show that the orbits under $Sp(\Lambda_1)\times Sp(\Lambda_2)$ of maximal totally isotropic
submodules $X$ of $\Lambda$ are characterized by the rank 
$\rk(\pi_1(X)/\rad(\pi_1(X))) =\rk(\pi_2(X)/\rad(\pi_2(X)))$, the invariants $d,d'$ of the $\rad(\pi_i(X))$, and
a Hecke double coset associated to paramodular groups of rank $\rk(\pi_i(X)/\rad(\pi_i(X)))/2$ derived from
$Sp(\Lambda_1), Sp(\Lambda_2)$. Translating that into 
matrix language we obtain the desired 
generalisation of Garrett's double coset decomposition.
Unfortunately, this requires a somewhat lengthy notation involving the
boundary components (associated to $\rad(\pi_1(X)),\rad(\pi_2(X))$ by
the results of Section 2) on which the paramodular groups of smaller rank act.
Important
special cases which look much simpler will be discussed in the remark following the proof. 
\begin{theorem}\label{maintheorem}
 Let $\Lambda_1,\Lambda_2$ be $\Z$-lattices of ranks $2m,2n$ on the
 vector spaces $V_1,V_2$ over $\Q$ with nondegenerate 
 alternating bilinear forms of square free levels $N_1,N_2$ and
 determinants $D_1,D_2$, let
 $N=\lcm(N_1,N_2)$ and $D=D_1D_2$, let $V=V_1\perp V_2,\Lambda=\Lambda_1 \perp
 \Lambda_2$. 

Let ${\mathcal B}=(e_1,\ldots,e_m,f_1,\ldots,f_m)$ be an ordered para-symplectic
basis for $\Lambda_1$ with $\langle e_i,f_j\rangle =d_i\delta_{ij},
d_i \mid d_{i+1}$ and  ${\mathcal B}'=(e'_1,\ldots,e'_n,f'_1,\ldots,f'_n)$  an ordered para-symplectic
basis for $\Lambda_2$ with $\langle e'_i,f'_j\rangle =d'_i\delta_{ij},
d'_i \mid d'_{i+1}$, let $v_i=d_i^{-1}f_i,
v_i'={d'_i}^{-1}f'_i$. Identify the elements of $Sp(V)$ with their
matrices with respect to the ordered symplectic basis
$(e_1,\ldots,e_m,e'_1,\ldots,e'_n,$ $v_1,\ldots,v_m$, $v'_1, \ldots,
v'_n)$ of $V$
and the elements of $Sp(V_1)$ with their
matrices with respect to the ordered symplectic basis 
$(e_1,\ldots,e_m,v_1,\ldots,v_m)$ of $V_1$,  the elements of $Sp(V_2)$ with their
matrices with respect to the ordered symplectic basis
$(e'_1,\ldots,e'_n,v'_1,\ldots,v'_n)$ of $V_2$.

Let $d\mid D_1,d'\mid D_2,r\le \min(m,n)$ satisfy $d\mid N_1^{m-r}, d'\mid
N_2^{n-r}, \frac{D_1}{d}\mid N_1^{r}, \frac{D_2}{d'}\mid N_2^{r}$, set
$u_1=m-r, u_2=n-r$.

Let $g_1=g_1({\mathcal B},u_1,d)$ be as in Lemma
\ref{matrixrepresentatives_cusps} with associated matrices
$\gamma_1=\gamma_1({\mathcal B},u_1,d) \in Sp_m(\Q), S_1=S_1({\mathcal
  B},u_1,d)\in SL_m(\Z)$ and define $g_2=g_2({\mathcal B}',
u_2,d')$ and $\gamma_2,S_2$ analogously for $\Lambda_2$, for $u_1=0$
set $S_1=S_1({\mathcal B},0,1)=1_m$, for $u_2=0$ set
$S_2=S_2({\mathcal B}',0,1)=1_n$.

Let $\tilde{v}_i=g_1v_i,
\tilde{f}_i=\tilde{d}_i\tilde{v}_i$, $\tilde{e}_i=g_1e_i$ and
$\tilde{v}'_i=g_2v'_i$, $\tilde{f}'_i=\tilde{d}'_i\tilde{v}'_i,
\tilde{e}'_i=g_2e'_i$ as in Lemma
\ref{matrixrepresentatives_cusps}.

Let
\begin{equation*}
  T=T(d,r)=\begin{pmatrix}
\tilde{d}_1&&\\&\ddots&\\&&\tilde{d}_r  
\end{pmatrix}, T'=T'(d',r)=\begin{pmatrix}
\tilde{d}'_1&&\\&\ddots&\\&&\tilde{d}'_r  
\end{pmatrix}.
\end{equation*}

Let $X^{(0)}= \bigoplus_{i=1}^m\Z e_i \oplus
\bigoplus_{i=1}^n\Z e_i'$ and let $P=\{g \in Sp(\Lambda)\mid g(\Q
X^{(0)})=\Q X^{(0)}\}$.

For a maximal totally
isotropic submodule $X$ of $\Lambda$ let
$\pi_i(X)=X_i$ for $i=1,2$ denote the projections
to $\Lambda_i$.

\medskip
Then for each 
double coset $(Sp(\Lambda_1)\times Sp(\Lambda_2))gP$ with $g \in
Sp(\Lambda)$  the maximal totally isotropic submodule
$X=gX^{(0)}$ of $\Lambda$ satisfies  $d(\rad(X_1))=d, d(\rad(X_2))=d'$ and
$\rk(X_i/\rad(X_i))=2r$ for $i=1,2$ for some triple of values $d,d',r$
as described above, and for each fixed such triple with associated
matrices $S_1,S_2$, 
a set of representatives of those double cosets $(Sp(\Lambda_1)\times
Sp(\Lambda_2))gP$ with $g \in Sp(\Lambda)$ for
which $X=gX^{(0)}$ has $d(\rad(X_1))=d, d(\rad(X_2))=d'$ and $X_i/\rad(X_i)$ has rank
$2r$ for $i=1,2$ is given by the matrices

\begin{equation*}
  \begin{pmatrix}
1_{m+n}&0_{m+n}\\
C&1_{m+n}    
  \end{pmatrix},
\end{equation*}

with
\begin{equation*}
 C=
  \begin{pmatrix}
    {}^tS^{-1}_1&0\\0&{}^tS^{-1}_2
  \end{pmatrix}
  \begin{pmatrix}
    0_{r,r}&0_{r,u_1}&{}^tB(T')&0_{r,u_2}\\
0_{u_1,r}&0_{u_1,u_1}&0_{u_1,r}&0_{u_1,u_2}\\
(T')B&0_{r,u_1}&0_{r,r}&0_{r,u_2}\\
0_{u_2,r}&0_{u_2,u_1}&0_{u_2,r}&0_{u_2,u_2}
  \end{pmatrix}
\begin{pmatrix}
  S^{-1}_1&0\\0&S^{-1}_2
\end{pmatrix},
\end{equation*}
where the matrices $\bigl(
\begin{smallmatrix}
  B&0\\
0&{}^tB^{-1}
\end{smallmatrix}\bigr)$ with $B \in GL_r(\Q)\cap M_r(\Z)$ run 
through a set of
representatives  of the double cosets
\begin{equation*}\Gamma^{(r)}(T')
  h \Gamma^{(r)}(T) \subseteq Sp_r(\Q),
\end{equation*}
for example the set ${\mathcal R_{d,d',r}}$ given in Corollary \ref{blockdiagonal_rep}.

\end{theorem}
\begin{proof} For $B$ as above the automorphism $\rho=\rho(r,d,d',B)$ of $V$ given by
  the matrix
  \begin{equation*}
                  \begin{pmatrix}
                     S^{-1}&0\\
                     0&{}^tS
                   \end{pmatrix}            
  \begin{pmatrix}
   1_{m+n} &0_{m+n}\\C&1_{m+n}
 \end{pmatrix}
                         \begin{pmatrix}
                           S&0\\
                           0&{}^tS^{-1}
                  
                         \end{pmatrix}
  = \begin{pmatrix}
   1_{m+n} &0_{m+n}\\{}^tSCS&1_{m+n}
 \end{pmatrix}
 \end{equation*}
with 
\begin{equation*}
  S=\begin{pmatrix}
  S_1^{-1}&0\\0&S^{-1}_2
\end{pmatrix},
\quad
 C=
 \begin{pmatrix}
    0_{r,r}&0_{r,u_1}&{}^tBT'&0_{r,u_2}\\
0_{u_1,r}&0_{u_1,u_1}&0_{u_1,r}&0_{u_1,u_2}\\
T'B&0_{r,u_1}&0_{r,r}&0_{r,u_2}\\
0_{u_2,r}&0_{u_2,u_1}&0_{u_2,r}&0_{u_2,u_2}
  \end{pmatrix}
\end{equation*}
with respect to the basis
$(e_1,\dots,e_m,e_1',\dots,e_n',v_1,\dots,v_m,v_1',\dots,v_n')$ of $V$
has matrix
\begin{equation*}
  \begin{pmatrix}
   1_{m+n} &0_{m+n}\\C&1_{m+n}
 \end{pmatrix}
\end{equation*}
with respect to the basis
$(\tilde{e}_1,\dots,\tilde{e}_m,\tilde{e}_1',\dots,\tilde{e}_n',\tilde{v}_1,\dots,\tilde{v}_m,\tilde{v}_1',\dots,\tilde{v}_n')$
of $V$ 
by definition of $S_1,S_2$. We have thus
\begin{eqnarray*}
  \rho(\tilde{e}_j)&=&\tilde{e}_j+\sum_{i=1}^r\tilde{d}_i'b_{ij}\tilde{v}_i'\\
                   &=&\tilde{e}_j+\sum_{i=1}^rb_{ij}\tilde{f}_I',\\
  \rho(\tilde{e}_j')&=&\tilde{e}_j'+\sum_{i=1}^rb_{ji}\tilde{d}_j'\tilde{v}_i\\
  &=&\tilde{e}_j'+\sum_{i=1}^r\tilde{d}_i^{-1}b_{ji}\tilde{d}_j'\tilde{f}_i.
\end{eqnarray*}
We notice that we have $\tilde{d}_i^{-1}b_{ji}\tilde{d}_j'\in \Z$ by
Corollary \ref{blockdiagonal_rep}, hence $ \rho(\tilde{e}_j' )\in
\Lambda_1$ for all $j$ and therefore $g \in Sp(\Lambda)$.
With $X=\rho(X^{(0)})$ we have then 
$X_1=\sum_{j=1}^m\Z\tilde{e}_j+\sum_{j=1}^r\Z(\tilde{d}_j'\sum_{i=1}^rb_{ji}\tilde{v}_i)$
with  
$\rad(X_1)=\sum_{i=r+1}^m\Z \tilde{e}_i$, $d(\rad(X_1))=d$.

In the same way we have
$X_2=\sum_{j=1}^n\Z\tilde{e}_j'+\sum_{j=1}^r\Z(\sum_{i=1}^rb_{ij}\tilde{f}'_i)$ with
$\rad(X_2)
=\sum_{i=r+1}^n\Z\tilde{e}_i'$,
$d(\rad(X_2))=d'$.

We can write 
$M(\rad(X_1))=\rad(X_1)+\sum_{i=r+1}^m\Z\tilde{f}_i,
\Lambda_1'=\sum_{i=1}^r(\Z\tilde{e}_i+\Z\tilde{f}_i)$ and have
$\Lambda_1=M(\rad(X_1))\perp \Lambda_1'$. In the same way we have 
$M(\rad(X_2))=\rad(X_2)+\sum_{i=r+1}^m\Z\tilde{f}'_i,
\Lambda_2'=\sum_{i=1}^r(\Z\tilde{e}'_i+\Z\tilde{f}'_i)$ and
$\Lambda_2=M(\rad(X_2))\perp \Lambda_2'$.

We can identify $X_1/\rad(X_1)$ with %
$\sum_{j=1}^r\Z\tilde{e}_j+\sum_{j=1}^r\Z(\tilde{d}_j'\sum_{i=1}^rb_{ji}\tilde{v}_i)\subseteq
\Lambda_1'$
and $X_2/\rad(X_2)$ with
$\sum_{j=1}^n\Z\tilde{e}_j'+\sum_{j=1}^r\Z(\sum_{i=1}^rb_{ij}\tilde{f}'_i)\subseteq
\Lambda_2'$ and obtain an isomorphism $\phi: X_1/\rad(X_1)\to
X_2/\rad(X_2)$ given by
$\phi(\tilde{e}_j)=\sum_{i=1}^rb_{ij}\tilde{f}_i'$,
$\phi(\tilde{v}_j)=\sum_{i=1}^ra_{ij}(\tilde{d}_i')^{-1}\tilde{e}_i'$ with
$A={}^tB^{-1}$.

We let now $W$ be a symplectic vector space of dimension $2r$ over
$\Q$ with symplectic basis $(x_1,\dots,x_r,y_1,\dots,y_r)$ and define
$\Sigma_1:\Q \Lambda_1'\to W_1,\Sigma_2:\Q \Lambda_2' \to W$ by
$\Sigma_1(\tilde{v}_j)=y_j,
\Sigma_1(\tilde{e}_j)=x_j$,$\Sigma_2(\tilde{e}_j')=d_jy_j,
\Sigma_2(\tilde{f}_j')=x_j$.
The matrix groups attached to $Sp(\Lambda_1'),Sp(\Lambda_2')$ with
respect to the $x_j,y_j$ are then $\Gamma^{(r)}(T), \Gamma^{(r)}(T')$
respectively and we have
\begin{eqnarray*}
  \Sigma_2 \circ \phi\circ
  \Sigma_1^{-1}(x_j)&=&\sum_{i=1}^rb_{ij}x_i,\\
  \Sigma_2 \circ \phi\circ
  \Sigma_1^{-1}(y_j)&=&\sum_{i=1}^ra_{ij}y_i
\end{eqnarray*} with $A={}^tB^{-1}$ as before, i.e., $\Sigma_2 \circ \phi\circ
  \Sigma_1^{-1}$ has matrix $\biggl(
  \begin{smallmatrix}
    B&0_r\\0_r&{}^tB^{-1}
  \end{smallmatrix}\biggr)$.

The previous proposition together with
the lemmata of this section leading up to it implies then the assertion.
\end{proof}
\begin{remark}
  \begin{enumerate}
  \item Assume $m=n, N_1=N_2, D_1=D_2$ and consider those double
    cosets $(Sp(\Lambda_1)\times Sp(\Lambda_2))gP$ with $r=m=n$, i.e.,
    the projections $X_1,X_2$ of $X=gX^{(0)}$ have zero radical. In
    the theorem we
    have then $T=T'=\diag(d_1,\ldots,d_m)$ and $S_1=S_2=1_m$, and the
    set of representatives of these double cosets consists in matrix
    notation of the 
    \begin{equation*}
      \begin{pmatrix}
        1_{m}&0_m&0_{m}&0_m\\
        0_m&1_m&0_m&0_m\\
        0_{m}&{}^tBT&1_{m}&0_m\\
TB&0_m&0_{m}&1_{m}
      \end{pmatrix},
    \end{equation*} where $B$ runs through a set of representatives of
    the double cosets 
    \begin{equation*}\Gamma^{(m)}(T)
  h \Gamma^{(m)}(T) \subseteq Sp_m(\Q).
\end{equation*}

\item    In the case that $\Lambda,\Lambda_1, \Lambda_2$ have level $1$ we
obtain (with $d=d'=D_1=D_2=1$ and the matrices $B$ being diagonal elementary divisor
matrices by known results for the Hecke algebra of $Sp_n(\Z)$) Garrett's
\cite{garrett} result with a coordinate free proof. That such a proof
should be possible has already been 
remarked in \cite{garrett}, see also \cite{murase}. Whereas in
Garrett's proof the relation between the representatives of the double
cosets $Pg(Sp_m(\Z)\times Sp_n(\Z))$ and the representatives of the
Hecke double cosets in $Sp_r(\Z)$ appears to be a coincidence, the
method chosen here shows that it is not. 
\end{enumerate}
\end{remark}

\section{Theta series for the paramodular group}\label{sec5}
We consider an $m$-dimensional vector space $V$ over $\Q$ with positive definite quadratic form
$Q:\: V \longrightarrow \Q$ and associated  symmetric bilinear forms
 \begin{equation*}
b(x,y) = Q(x+y)-Q(x) - Q(y),\, B(x,y) = \frac{1}{2}b(x,y).
 \end{equation*}
The Gram matrix of an $n$-tuple $(x_1,\ldots,x_n) \in V^n$ of vectors in $V$ with respect to
$Q$ is the matrix $Q(x_1,\ldots,x_n) = (B(x_i,x_j))_{1 \leq i,j \leq n}$.

For a lattice $L=\bigoplus_{j=1}^{m} \Z e_j$ of full rank on $V$ with basis $(e_1,\ldots,e_m)$ the
dual lattice is
$L^{\#} = \{y \in V~|~b(y,L) \subseteq \Z\}$, the level $N(L)$ is the smallest $N\in\N$ with 
$NQ(L^{\#}) \subseteq \Z$ and the discriminant ${\rm disc}(L)$ is $\det(b(e_i,e_j))$. 

The lattice is integral if $Q(L) \subseteq \Z$, it is unimodular if $L=L^{\#}$ (this corresponds
to an even unimodular lattice in the notation of \cite{OM,CS}). Lattices $L$ and $K$ on $V$ are in the
same class if there is 
 \begin{equation*}
\varphi \in O(V) = \{\varphi \in {\rm GL}(V)~|~ Q(\varphi(x)) = Q(x) \mbox{ for all }
 x \in V\}
 \end{equation*}
with $\varphi(K) = L$, they are in the same genus ($K \in {\rm gen}(L)$) if for all primes $p$
there is $\varphi_p \in O(V \otimes \Q_p)$ with $\varphi_p(K \otimes \Z_p) = L\otimes \Z_p$.

Similarly, two $n$-tuples $(K_1,\ldots,K_n)$, $(L_1,\ldots,L_n)$ of lattices on $V$ are in the
same class if there exists $\varphi \in O(V)$ with $\varphi(K_i) = L_i$ for 
$1 \leq i \leq n$, similarly for the genus. We write
 \begin{equation*}
 O(K_1,\ldots,K_n) = \{\varphi \in O(V)~|~ \varphi(K_i) = K_i \mbox{ for } 1 \leq i \leq n\}
 \end{equation*}
and $O_{\A}(K_1,\ldots,K_n)$ for its adelization, so that $\varphi \in O_{\A}(K_1,\ldots,K_n)$ is a
tuple $(\varphi_p)_{p\in {\mathbb P}\cup \{\infty\}}$ with $\varphi_p \in O(V \otimes \Q_p)$ and 
$\varphi_p(K_i \otimes \Z_p) = K_i \otimes \Z_p$ for all $p$ including $p = \infty$. With 
this notation the classes in the genus of the $n$-tuple  $(K_1,\ldots,K_n)$ correspond to the double 
cosets 
$ O(V) \varphi O_{\A}(K_1,\ldots,K_n)$ 
in the adelic orthogonal group $O_{\A} (V)$ of the quadratic space $(V,Q)$.
Assuming $K_1\supseteq K_2 \supseteq \cdots \supseteq K_n \supseteq MK_1$ for some $M \in  \N$ 
we see that $O(K_1,\ldots,K_n)$ is a congruence subgroup of $O(K_1)$,
hence in particular of finite index.
Since the number of double cosets $O(V)\varphi O_{\A}(K_1)$, being the
number of classes in the genus of $K_1$, is known to be finite,
the 
number of classes in the genus of $K_1,\ldots,K_n$ is finite too. This assumption can be made
without loss of generality since there exist $M_2,\ldots,M_n \in \N$ with $M_{i+1}K_{i+1} \subseteq M_iK_i
\subseteq K_i$ for all $i$ and $M$ with $MK_1 \subseteq K_n$.

\begin{definition}
  Let $L_1,\dots, L_n$  be lattices on $V$.

  The theta series
   $\vartheta^{(n)}(L_1,\dots,L_n)$ is the function on ${\mathfrak
     H}_n$ given for $Z\in \Hn$ by 
   \begin{equation*}
     \vartheta^{(n)}(L_1,\dots,L_n;Z)=\sum_{x_1 \in L_1,\dots,x_n\in
       L_n}\exp(\pi i \tr(Q(x_1,\dots,x_n)Z)).
   \end{equation*}
  \end{definition}
 \begin{remark}
     \begin{enumerate}
     \item for $L=L_1=\dots=L_n$ we obtain the usual degree (or genus)
       $n$ theta series of the lattice $L$.
\item We obtain a matrix notation for our theta series by fixing a
  basis $(e_1,\ldots,e_m)$ of $V$ and matrices $U_1,\ldots U_n \in
  GL_m(\Q)$ such that the coordinate vectors with respect to the given
  basis of vectors in $L_i$ run through $U_i\Z^m$. If $S$ is the Gram
  matrix of $Q$ with respect to the given basis we obtain
  \begin{equation*}
    \vtn(L_1,\dots,L_n;Z)=\sum_G\exp(\pi i \tr(S[UG]Z)),
  \end{equation*}
where $G=(g_1,\dots,g_n)$ runs over the integral $(m\times
n)$-matrices and where we write $UG=(U_1g_1,\dots,U_ng_n)$. 
     \end{enumerate}
   \end{remark}
   \begin{theorem}
The theta series satisfy the transformation formula
\begin{equation*}
  \vtn(L_1^\#,\dots,L_n^\#,-Z^{-1})=\sqrt{\det(Z/i)}^m\,\prod_{j=1}^n\sqrt{{\rm disc}(L_j)}\,\,\vtn(L_1,\ldots,L_n;Z),
\end{equation*}
where  $\sqrt{\det(Z/i)}$ is defined as in \cite[Hilfssatz
0.10]{freitagbook}, i.e., it is continuous on $\Hn$ and has positive
real values for $Z=iY$ on the imaginary axis. 
   \end{theorem}
   \begin{proof}
     By using the matrix formulation of our theta series we can
     proceed as in the proof of \cite[Hilfssatz 0.12]{freitagbook} by
     computing the Fourier coefficient $a(H)$ at a matrix $H=(h_1,\dots,h_n)\in \Z^{m,n}$ of
     the periodic function $f$ on $\C^{m,n}$ given by
     \begin{equation*}
       f(W)=\sum_G\exp(\pi i \tr(S[UG+UW]Z)).
     \end{equation*}
We arrive then in the same way as there at
\begin{equation*}
  a(H)=\exp(-\pi i \tr(S^{-1}[{}^t\,U^{-1}H]Z^{-1})) \int\exp(\pi i
  \tr(S[UV]Z)) dV
\end{equation*}
where $V$ runs over $\R^{m,n}$ and where we write
${}^t\,U^{-1}H=({}^t\,U_1^{-1}h_1,\dots,{}^t\,U_n^{-1}h_n)$.
Applying the transformation formula for integrals and using that the
coordinate vectors with respect to the given basis of the dual lattice
$L_i^\#$ run over $U_i^{-1}\Z^m$ we obtain the final result as in \cite{freitagbook}. 
   \end{proof}
   \begin{theorem}\label{theta_transformation}
 Let $T\in M_n(\Z)$ be an elementary divisor matrix with diagonal
 entries $1=t_1,t_2,\ldots,t_n$, let $L_j$ for $1\le j \le n$ be
 positive definite even
 $t_j$-modular lattices (so $L_j^\#=t_j^{-1}L_j$ and $Q(L_j)=t_j \Z$)
 of rank $m$ with $L_1\supseteq L_2\supseteq
 \dots \supseteq L_n$ (we will call such a chain of lattices 
 paramodular of level $T$ in the sequel).

Then $\vtn(L_1,\ldots,L_n)$ is a modular form of weight $k=m/2$ for
the paramodular group $\Gamma^{(n)}(T)$.
   \end{theorem}
   \begin{proof}
 By Satz 1.12 of \cite{kapplerdiss} the group $\Gamma^{(n)}(T)$ is
 generated by the matrices 
 \begin{equation}\label{J_T}
 J_T=
 \begin{pmatrix}
   0_n&-T^{-1}\\
T&0_n
 \end{pmatrix},\quad
 \begin{pmatrix}
   1_n&t_i^{-1}E_{ii}\\
0_n& 1_n
 \end{pmatrix},\quad
 \begin{pmatrix}
   1_n&t_i^{-1}(E_{ij}+E_{ji})\\
0_n&1_n
 \end{pmatrix} (1\le i<j\le n),
 \end{equation}
where $E_{ij}$ denotes the $n\times n$-matrix with entry $1$ in
position $(i,j)$ and $0$ in all other positions.

We check the transformation behavior under these generating matrices:

By the previous theorem we have
\begin{align*}
  \vtn(L_1,&\ldots,L_n;J_TZ)=\vtn(L_1,\ldots,L_n;-(TZT)^{-1})\\
&=\sqrt{\det(Z/i)}^m(\det(T))^m\prod_{j=1}^n\sqrt{{\rm
    disc}(L_j^\#)}\vtn(L_1^\#,\ldots,L_n^\#;TZT) \\
                      &=\sqrt{\det(TZ/i)}^m\vtn(t_1L_1^\#,\ldots,t_nL_n^\#; Z)\\
           &=\sqrt{\det(TZ/i)}^m\vtn(L_1,\ldots,L_n;Z)\\
             &={\det(TZ)}^{\frac{m}{2}}\vtn(L_1,\ldots,L_n;Z)
\end{align*}
since the rank $m$ of a positive definite even unimodular lattice is divisible by $8$.
Moreover, the translation matrices  
$\begin{pmatrix}
   1_n&t_i^{-1}E_{ii}\\
0_n& 1_n
 \end{pmatrix},\quad
 \begin{pmatrix}
   1_n&t_i^{-1}(E_{ij}+E_{ji})\\
0_n&1_n
 \end{pmatrix}$ leave $\vtn(L_1,\ldots,L_n)$ invariant since by
 assumption we have $t_i\mid Q(x_i)$ and $t_i \mid B(x_i,x_j)$ for
 $1\le i<j\le n$ and 
 $x_i\in L_i$. 
\end{proof}
\begin{remark}
  \begin{enumerate}
    \item As noticed in Remark \ref{remark_para} it is no restriction
      of generality to assume $t_1=1$ for a diagonal
      $T=\diag(t_1,\ldots,t_n)$, we will hence do so in the sequel.
  \item If we permute the lattices $L_i$ by some $\sigma \in S_n$ and
    apply the same permutation to the diagonal entries of the matrix
    $T$ we
    obtain a paramodular form with respect to this permuted matrix,
    hence with respect to a conjugate paramodular group. We
    will use this obvious fact later.
 \item If we relax the conditions on the $L_i$ by demanding
   $t_j$-modularity only for the $p$-adic completions of the $L_j$ for
   all $p$ deviding $t_n$ we  obtain modular forms for
groups $\Gamma_0^{(n)}(T,N)$ (as mentioned in Section \ref{sec2})
which are sort of a mix between paramodular groups and groups 
of type $\Gamma_0^{(n)}(N)$. The details of this will not be worked out here.  
 \end{enumerate}
\end{remark}
 \section{Siegel's main theorem}\label{sec6}
\begin{definition}
   Let $L_1\supseteq L_2\supseteq\dots\supseteq L_n$ be lattices on $V$. 
   \begin{enumerate}
\item The weight $w(\gen(L_1,\dots,L_n))$ of the genus of
  $(L_1,\ldots,L_n)$ is given as 
  \begin{equation*}
    \sum_{(M_1,\ldots, M_n)}\frac{1}{\vert O(M_1,\ldots,M_n)\vert},
  \end{equation*}
where the summation is over a set of representatives of the classes in
the genus of $(L_1,\ldots,L_n)$.
\item The genus theta series $\vartheta^{(n)}(\gen(L_1,\ldots,L_n))$ is
  given as
  \begin{equation*}
     \vartheta^{(n)}(\gen(L_1,\ldots,L_n))=\frac{1}{w(\gen(L_1,\ldots,L_n))}
     \sum_{(M_1,\ldots, M_n)}\frac{\vartheta^{(n)}(M_1,\ldots,M_n)}{\vert
       O(M_1,\ldots,M_n)\vert}, 
  \end{equation*}
where the summation is over a set of representatives of the classes in
the genus of $(L_1,\ldots,L_n)$.
   \end{enumerate}
 \end{definition}
 \begin{theorem}
    Let $k \equiv 0 \bmod 4$, and let $T$ be an elementary divisor
    matrix of square free level.
    Then
    there is exactly one genus of $n$-tuples $(L_1,\ldots,L_n)$ of
    lattices of rank $m=2k$ which are paramodular of level $T$.  
 \end{theorem}
 \begin{proof}
 Let $L_1$ be an even unimodular lattice of rank $m$. It is well known
 and can be seen using the corollaries on p. 116 and 119 of \cite{cassels}
 that the $p$-adic completion $(L_1)_p$ of $L_1$ is an orthogonal  sum of
 hyperbolic planes for any prime $p$, and it is then obvious that it
 contains a sublattice (of index $p^k$) that is $p$-modular. We can
 therefore obtain a chain of sublattices $L_i\supseteq L_{i+1}$ of
 $L_1$ for which each $L_i$ is $t_i$-modular, which settles the
 existence claim.

On the other hand, given two such chains of lattices $L_i, K_i$, we
can assume the lattices to be on the same rational quadratic space
since there is only one genus of even unimodular lattices of given
rank. If $p$ is a prime dividing $t_n$, the completion $(L_n)_p$ is a
$p$-modular sublattice of the maximal lattice $(L_1)_p$, which is an
orthogonal sum of hyperbolic planes. By \cite[Satz 9.5]{eichlerbook} there is
a hyperbolic basis $e_1,\ldots,e_k,f_1,\ldots,f_k$ (i.e., the
$e_i,f_i$ are isotropic vectors with $B(e_i,f_j)=\delta_{ij}$), such
that $e_1,\ldots,e_k,pf_1,\ldots,pf_k$ is a basis of $(L_n)_p$. In the
same way we obtain analogous bases $e_1',\ldots,e_k',f'_1,\ldots,f_k'$
of $K_1$ and $e_1',\ldots,e_k',pf'_1,\ldots,pf_k'$ of $K_n$. Sending
$e_i$ to $e_i'$ and $f_i$ to $f_i'$ we obtain a local isometry at $p$ mapping
$L_i$ onto $K_i$ for $1\le i \le n$, so the two chains are in the same genus.

\end{proof}
\begin{remark}
 An arithmetic study
   of the classes in this genus will be quite interesting. For
   example, for a prime $p$ and $T=\biggl(
   \begin{smallmatrix}
     1_a&0\\0&p1_b
   \end{smallmatrix}\biggr)$ the lattice chains are of the type
   $(L,\dots,L,K,\dots,K)$ with $a$ copies of the even unimodular
   lattice $L$ and $b$ copies of the even $p$-modular lattice $K$.
   
   Two such chains  $(L,\dots,L,K,\dots,K)$ and  $(L',\dots,L',K',\dots,K')$
   belong to the same class if $L,L'$ are in the same class and (with
   $L=L'$) the sublattice $K'$ of $L$
   is in the orbit of $K$ under the action of the group of
   automorphisms of $L$, so the number of classes in this genus of 
   chains with first entry in the class of $L$ is the number of these orbits.
\end{remark}
\begin{lemma}
For $T$ of square free level and $k>n+1$ the space of Siegel
Eisenstein series of weight $k$ for $\Gamma^{(n)}(T)$ has dimension $1$.  
\end{lemma}
\begin{proof}
  We recall that Siegel Eisenstein series, i.e. Eisenstein series associated
  to zero-dimensional cusps, of weight $k$ define a space of dimension equal to
  the number of equivalence classes of zero-dimensional cusps.
  The claim follows from  Corollary \ref{cusps_cosets_isotropic}.   
\end{proof}

 \begin{theorem}[Siegel's main theorem for the paramodular group]
  Let $T$ be an elementary divisor matrix of square free level and let
  $\gen((L_1,\ldots,L_n))$ be the unique  genus of $n$-tuples of
  positive definite 
    lattices of rank $m=2k$ (with $8\mid m$) which are paramodular of
    level $T$. Assume $k>n+1$ and define by
    \begin{equation*}
E^{(n),T}_k(Z):=\sum_{g \in P\backslash \Gamma^{(n)}(T)}j(g,Z)^{-k}
\end{equation*}   
the unique normalized (Siegel) Eisenstein series which is paramodular of
    level $T$ (where $j(
    \bigl(\begin{smallmatrix}
      A&B\\C&D
    \end{smallmatrix}\bigr),Z)=\det(CZ+D)$ as usual and where
    $P$ as in Theorem \ref{maintheorem} is the subgroup of 
    matrices with upper triangular block decomposition of the
    paramodular group $\Gamma^{(n)}(T)$). 

Then one has 
\begin{equation}
   \vartheta(\gen(L_1,\ldots,L_n))=E^{(n),T}_k.
\end{equation}
 \end{theorem}
 \begin{proof}
This is a consequence of the general Siegel-Weil theorem as given in
\cite{kudla_rallis}.
The translation between the adelic setting used there and our present
classical setting is provided by 
\cite{kudla_extensions}, see in particular Section IV.2. In the notation used there we take as
the test function $\varphi_p:V_p^n\to \C$ at the finite primes $p$ the
characteristic function of the
$p$-adic completions of the given $n$-tuple of lattices and let
$\varphi_\infty(x_1,\dots,x_n)=\exp(-\pi \tr(Q(x_1,\dots,x_n)))$.
A standard argument (see e.g. \cite[Section 2]{yoshida}) shows that  the integral $I(g;\varphi)$ is the adelic
function corresponding to Siegel's weighted
average  $\vartheta^{(n)}(\gen(L_1,\ldots,L_n))$. 

That  the adelic Eisenstein series $E(g,s_0,\Phi)$ corresponds to the
classical Eisenstein series above under the usual correspondence
between adelic and classical modular forms is checked as in the case
of level $N=1$ in
\cite{kudla_extensions}, replacing $\Gamma=Sp(n,\Z)$ by the
paramodular group $\Gamma=\Gamma^{(n)}(T)$. Notice (still using
Kudla's notation) for this that Kudla's argument for the case $Sp(n,
\Z)$ goes through unchanged for our situation if we can show that 
$\Gamma=\Gamma^{(n)}(T)$ satisfies $P(\Q)\Gamma=G(\Q)$, thus
$P(\Q)\backslash G(\Q)=P(\Q)\cap \Gamma\backslash \Gamma$, and that the
function $\Phi_f$ associated to  the $\varphi_p$ satisfies $\Phi_f(\gamma)=1$ for all $ \gamma \in
\Gamma^{(n)}(T)$.
Indeed Corollary \ref{cusps_cosets_isotropic} implies $P(\Q)\Gamma=G(\Q)$.
Moreover, using the generators discussed in the proof of Theorem
\ref{theta_transformation} and the usual formulas for the action of
the local Weil representation $\omega_p$ (see again e.g. \cite[Section
2]{yoshida}), it is easily
checked that $\omega_p(\gamma)\varphi_p=\varphi_p$ for all $\gamma \in
\Gamma^{(n)}(T)$ and all primes $p$ and hence 
$\Phi_f(\gamma)=1$ for all $ \gamma \in
\Gamma^{(n)}(T)$ holds.
We notice in passing
that the last argument can also be used to give an adelic proof of Theorem
\ref{theta_transformation}.
 \end{proof}
 \begin{remark}
Using the relation
\begin{equation*}
  \Gamma^{(n)}(T)=
  \begin{pmatrix}
    T^{-1}&0_n\\
    0_n&T
  \end{pmatrix}
  \begin{pmatrix}
    T&0_n\\0_n&1_n
  \end{pmatrix} \hat{\Gamma}^{(n)}(T)
  \begin{pmatrix}
    T^{-1}&0_n\\0_n&1_n
  \end{pmatrix}
  \begin{pmatrix}
    T&0_n\\
    0_n&T^{-1}
  \end{pmatrix}
\end{equation*}
one sees that
\begin{equation*}
  E^{(n),T}_k(T^{-1}ZT^{-1})=\hat{E}^{(n),T}(Z),
\end{equation*}
where $\hat{E}^{(n),T}(Z)$ denotes the Eisenstein series of Siegel in 
\cite{siegel_stufe} 
 defined using the group $\hat{\Gamma}^{(n)}(T)$, so that Siegel's
 Eisenstein series is the genus theta series attached to  the lattice
 $n$-tuple  $(L_1^\#,\ldots,L_n^\#)$.    
 \end{remark}
\section{The basis problem for cusp forms of square free level}\label{sec7}
\subsection{Pullback of Eisenstein series}\label{sec7.1}
We start from a ``polarization matrix'' ${\mathcal T}$ of size $2n$
of the special form 
$\mathcal T= \left(\begin{array}{cc} T & 0\\
0 & T\end{array}\right)$. Furthermore we denote the element 
$\left(\begin{array}{cc} z & 0\\ 0 & w\end{array}\right)\in {\mathfrak H_{2n}}$ with $z,w\in {\mathfrak H}_n$ by $\iota(z,w)$.
\\
For the degree $2n$ Eisenstein series $E^{(2n),\mathcal T}_k$
and a degree $n$ cusp form $f$ for $\Gamma^{(n)}(T)$ we consider
the
function $\Omega(f)$ on $H_n$ defined by

$$w\longmapsto \Omega(f)(w):= 
\int_{\Gamma^{(n)}(T)\backslash H_n} f(z) 
\overline{E_k^{(2n),\mathcal T} 
(
\iota(z,-\bar{w}))
}\det(y)^k dz_n,$$ 
where $dz_n$ denotes the invariant symplectic volume element on 
${\mathfrak H}_n$.\\
We only give a sketch of the unfolding in this case and 
refer to \cite{garrett} and \cite{bs1},
where the case of level 1 was considered in detail. 
We recall from loc. cit. that one needs a ``Garrett double coset decomposition''
and a ``twisted coset decomposition'' describing $P\backslash \Gamma^{2n}({\mathcal T})/(\Gamma^{(n)}(T)\times \Gamma ^{(n)}(T))$.\\
Once the former decomposition 
is available, then the latter decomposition is obtained by
a routine calculation.\\
The key new ingredient is  Theorem \ref{maintheorem},
applied to the situation $n=m, T=T'$, which 
gives the desired double coset decomposition. Noticing that in view
of $\Gamma^{(n)}(T) B \Gamma^{(n)}(T)=\Gamma^{(n)}(T) (-B) \Gamma^{(n)}(T)$ the
representatives given there can also be used as representatives for the double
cosets $Pg (\Gamma^{(n)}(T)\times
\Gamma^{(n)}(T))\subseteq \Gamma^{(2n)}({\mathcal T})$. 
We can then closely follow
the strategy of loc.cit.:

Without further notice, we freely use the notation from Theorem
\ref{maintheorem}. Note that due to the definition of Siegel Eisenstein series
we have to use the double coset decomposition of Theorem \ref{maintheorem}
in ``opposite order" (by inverting everything).
We decompose the  summation defining the Eisenstein series 
into subseries according to Theorem \ref{maintheorem}:

\smallskip
A natural decomposition is given by collecting all 
left cosets  $P\cdot g\in \Gamma^{(2n)}({\mathcal T})$,
which belong to a fixed double coset (encoded by $C$):
$$E_k^{(2n),{\mathcal T}}(\iota(z,w))=\sum_C \omega_C(z,w).$$
Note that $\omega_C$ is a modular form in $z$ and $w$ for $\Gamma^{(n)}(T)$.
The summands appearing in a fixed $\omega_C$  are described 
by 
the ``twisted double coset decomposition''; as in the level one 
case, this is of type
\begin{equation}P\cdot \left(\begin{array}{cc} 1_{2n} & 0\\
C & 1_{2n}\end{array}\right) (\Gamma^{(n)}(T)\times \Gamma^{(n)}(T))=\cup_{\gamma,\delta} P\cdot \left(\begin{array}{cc} 1_{2n} & 0\\
C & 1_{2n}\end{array}\right) \cdot \gamma\times \delta, 
\label{twisted}\end{equation}  
where $\gamma$ and $\delta$ run over representatives of left cosets in $\Gamma^{(n)}(T)$ modulo certain subgroups depending on $C$, but independently 
of each other. This follows in the same way as in the level one case.

\smallskip
We can compute the contributions of the 
$\omega_C$  to the integral $\Omega(f)$
individually: The main points are
\begin{itemize}
\item By elementary matrix calculation, one has
\begin{equation}j(\left(\begin{array}{cc} 1_{2n} & 0_{2n}\\
C & 1_{2n}\end{array}\right),\iota(z,w))= 
\det(1_r-{} ^tBT'^{-1}w^S_1T'^{-1}B z^S_1),\label{kern}\end{equation}
where $w^S_1$ denotes the submatrix of size $r$ in the upper left corner
of $S_2^{-1}wS^{-t}_2$ (and analogously, $z^S_1$ = upper left corner
of size $r$  in $S_1^{-1}zS_1^{-t}$.
\item 
None of the $\omega_C$ with $r<n$ contributes to the doubling integral:\\
One may argue in essentially the same way as for level one (see \cite[\S 8]{garrett}
or \cite[p.238]{KliKatata}.
We briefly sketch the proof in our context:\\
It follows from general principles that  $\omega_C$ as a 
function of $z$ is orthogonal to cusp forms:
This holds quite generally (under suitable convergence conditions)
for any modular form $F$ on ${\mathfrak H}_n$ 
for an arithmetic subgroup $\Gamma$ of 
$Sp(n,{\mathbb Q})$, constructed by averaging from a
function $\phi$ on $H_r$ with $r<n$ by
$$F(Z):=\sum_\gamma \Phi\mid_k \gamma.$$  
Here $\Phi(Z):=\phi(z_1)$ with $z_1$ as explained above and
$\gamma$ runs over $\Gamma$ modulo an appropriate subgroup
(a subgroup of a Klingen parabolic, containing nontrivial translations). 
The requested orthogonality follows 
by a standard unfolding argument, see e.g. \cite[\S 7]{Klibuch}.
We just have to observe, using Theorem 4.8 and 
(\ref{twisted}) together with (\ref{kern}) that 
$\omega_C$ can be viewed as such a function $F$ if $r<n$;
we do not need any explicit description of $\gamma,\delta$ in (\ref{twisted})
for this. 

\item In the special case $r=n$ with $S_1=S_2=1_n$, $T'=T$ and 
$B=1_n$ the right hand side of (\ref{kern}) is just $\det(1_n-TwTz)$.

The reproducing formula (attributed to Selberg)
$$\int_{{\mathfrak H}_n} f(z) \overline{\det(z-\bar{w})}^{-k} \det(y)^kdz_n=a_{n,k} f(w),$$
valid for $k>2n$ and any holomorphic function $f$ on ${\mathfrak H}_n$ 
satisfying a suitable
growth condition (where $a_{n,k}$ is a nonzero constant, for the explicit value
see e.g. \cite[p.78]{Klibuch})  gives then for a cusp
form $f$- after unfolding - for this special $C$
\begin{eqnarray*} \int_{\Gamma_0(T)\backslash {\mathfrak H}_n} f(z) 
\overline{\omega_C (
z,-\bar{w}))} \det(y)^k dz_n&=&\int_{{\mathfrak H}_n} f(z) \overline{\det(1+T\bar{w}Tz)}^{-k} \det(y)^kdz_n\\
&=& a_{n,k} \det(T)^{-k} \left(f\mid_k J_T\right))(w)\\
&=&a_{n,k} \det(T)^{-k}f(w).
\end{eqnarray*}
Here $J_T$ is as in \ref{J_T}.
\item More generally (and again just as in the level one case), 
for arbitrary $B$ with $r=n$, the Hecke operator %
associated to the double coset $\Gamma_n(T) h\Gamma_n(T)$ comes in:
 
\begin{eqnarray*}
\lefteqn{\int_{\Gamma_0(T)\backslash {\mathfrak H}_n} f(z) 
\overline{
\omega^B(z,-\bar{w})  }\det(z)^k dz_n=}\\
& & a_{n,k} \det(T)^{-k}\left(f\mid \Gamma^{(n)}(T)h\Gamma^{(n)}(T)\right)(w)
\cdot \det(h)^{-k}.\end{eqnarray*}
Here we switched notation from $\omega_C$ to $\omega^B$.
\end{itemize}
In the formula above, we used the standard definition of Hecke operators:
For any double coset $\Gamma^{(n)}(T)g\Gamma^{(n)}(T)$ with $g\in Sp(n,{\bf Q})$, 
we get an endomorphism of the space of modular forms (cusp forms) by
$$\left(f\mid \Gamma^{(n)}(T)g\Gamma^{(n)}(T)\right)(z):=\sum_{\delta}
(f\mid_k\delta)(z),$$
where  $\Gamma^{(n)}(T)g\Gamma^{(n)}(T)=\cup\,\Gamma^{(n)}(T)\delta$.
Note that the algebra generated by these endomorphisms is in general 
not commutative.
The results from Section \ref{sec3} show, however,  that commutativity holds
for the subalgebra generated by the $T(m)$. Moreover, the $T(m)$ define
selfadjoint operators (w.r.t. the Petersson inner product) on the 
space of cusp forms; to see this, one can use the same kind of reasoning
as for the commutativity in Section \ref{sec3}, using the involution on the Hecke algebra
induced by $\alpha\longmapsto \alpha^{-1}$.
In particular, the space of cusp forms has a basis consisting of
simultaneous eigenforms of all the $T(m)$.
   
\subsection{Nonvanishing}\label{sec7.2}

Now we assume that $f$ is an eigenform of all the Hecke operators $T(m)$ 
with eigenvalues $\lambda(m)$.
Then $\Omega(f) $ is proportional to $f$ with a factor, 
which equals - up to the factor $a_{n,k}\det(T)^{-k}$ - the Dirichlet series
$${\mathcal L}(f,s):=\sum_d \lambda(d) d^{-s}$$
at $s=k$.

\smallskip
We will show that this Dirichlet series does not vanish at $s=k$:

We observe that this series inherits the absolute convergence at $s=k$  
from the corresponding property of the degree $2n$ Eisenstein series 
provided that $k>2n+1$.
\\
Also, from Lemma 3.8.  we have an Euler product expansion (for
$\Re(s)\gg 0$)

$${\mathcal L}(f,s)=\prod_p {\mathcal L}_p(f,s)$$

and it suffices to show that all these Euler 
factors are different from zero at $s=k$.
We shall do this without explicitly determining the Euler factors (for
primes not dividing $\det(T)$ one can write them as standard Euler factors
expressed by Satake parameters, but we do not use this here).\\

The Euler factors are of the form

$${\mathcal L}_p(f,s) =\sum_{j=0}^{\infty} \lambda(
p^j) p^{-js}=1+
{\mathcal R}_p(f,s).
$$

We show that the subseries ${\mathcal R}_p(f,s)$ of ${\mathcal L}_p(f,s)$
defined by $\mathcal R_p(f,s)=\sum_{j=1}^{\infty} \lambda_p(p^j)p^{-js}$ 
is of absolute value smaller than 1 at $s=k$ if $k$ is large enough:

The consideration in \cite{koh} 
shows that Hecke eigenvalues of cusp forms may 
quite generally be estimated by the number
of left cosets in the double coset defining a Hecke operator.

The results from Lemma 3.6 and Lemma 3.7 allow us then to
estimate $\lambda(p^j)$ by a power of 
$p^j$. 

We observe that the condition

$$\sum_{j=1}^{\infty} p^{-js} = \frac{p^{-s}}{1-p^{-s}} < 1$$

 holds for all real $s\geq  1$ if $p$ is odd (and for $p=2$ it holds for $s>1$). 

Using the explicit estimates from Lemma 3.6. we see that ${\mathcal L}_p(s)$
does not vanish at $s=k$ provided that $k\geq 2n+2$ and $p$ is odd;
the case $p=3$ needs a minor additional consideration.

Furthermore the nonvanishing also holds for $p=2$  if 
$k>2n+3$.
\begin{remark} The nonvanishing of ${\mathcal L}_p(s)$ at $s=k$
for $p$ coprime to $\det(T)$ also follows (under somewhat weaker conditions)
from the explicit form of the Euler factor, which is of the same type as in the 
level one case.\end{remark}
\begin{proposition}\label{nonvanish}  Assume that $k$ is even with $k> 2n+2$  and $f$ is a 
cusp form of weight $k$ for $\Gamma_0^{(n)}(T)$ and also a Hecke eigenform for 
all operators $T(d)$.
Then ${\mathcal L}(f,s)$ is nonzero at $s=k$, in particular,
$\Omega(f)$ is nonzero. If $\det(T)$ is odd, then this holds for $k\geq 2n+2$.
\end{proposition}
\subsection{ Basis problem}\label{sec7.3}
To combine the considerations above about pullbacks of Eisenstein series  
with Siegel's theorem, we
have to change our setting slightly (this is mainly a matter of notation):\\
We assume now that $T$ is of elementary divisor form:
$$T=\diag(1,t_2,\dots, t_n)\qquad (t_i\mid t_{i+1}).$$
We let $(L_1, L_2,L_3,\dots,L_n)$ run over representatives of the classes
in the genus of $n$-tuples of lattices of rank $m=2k$ which are paramodular of level $T$
(see Theorem 6.2).
Then 
$(L_1,L_1,L_2,L_2, ... , L_n,L_n)$ runs over representatives of the classes 
in the genus of 2n-tuples of lattices of the same rank with paramodular level
$\diag(1,1,t_2,t_2,\dots , t_n,t_n).$ 
By applying a suitable permutation of the entries of 
$Z\in {\mathfrak H}_{2n}$ we may now reformulate Siegel's theorem: The Eisenstein series
$E_k^{(2n),{\mathcal T}}$ is a linear combination of theta series
$\vartheta^{(2n)}(L_1,L_2,\dots , L_n, L_1,\dots ,L_n)$,
in particular,
$E_k^{(2n),{\mathcal T}}(\iota(z,w))$ is a linear combination of
$$\vartheta^{(n)}(L_1,\dots , L_n)(z)\times \vartheta^{(n)}(L_1,\dots, L_n)(w).$$
We may therefore express $\Omega(f) $ in the usual way as a linear 
combination of theta series of the desired type. 
Taking into account that $\Omega$
defines an automorphism of the space of cusp forms 
(see Proposition \ref{nonvanish}), we obtain

\begin{theorem}[Basis problem for paramodular cusp forms]
For $k\geq 2n+4$ and $4\mid k$ all cusp forms for $\Gamma^{(n)}(T)$ 
are linear combinations of theta series of type $\vartheta^{(n)}(L_1,\dots,L_n)$
as described in Section \ref{sec5}. If $\det(T)$ is odd, the same holds for $k\geq 2n+2$.
\end{theorem}
\begin{remark} Using standard techniques about equivariant holomorphic 
differential operators \cite{ibudiff}, one can deduce in the 
same way statements concerning 
theta series with harmonic polynomials
and  one can also give a solution of the basis problem  
for cuspidal vector-valued modular forms, see e.g. \cite{bks}. 
\end{remark}

 \parbox[t]{5.3cm}{ Siegfried Böcherer, Kunzenhof 4B\\79117 Freiburg, 
   Germany \\boecherer@t-online.de}
 \hfill \parbox[t]{7cm}{Rainer  Schulze-Pillot\\Fachrichtung
   Mathematik\\ Universität des
Saarlandes, Postfach 151150, 66041 Saarbrücken,
Germany\\schulzep@math.uni-sb.de}

\end{document}